\documentclass[10pt]{amsart}
\usepackage{amssymb,latexsym}
\numberwithin{equation}{section}
\def\PP#1{{\mathbb P}^{#1}}
\def\PPP{\mathbb P}

\def\GG{\mathbb G}

\def\G{\vec{G}}

\def\max{\mathrm{max}}

\def\sr{s\mathrm{rk}}

\def\a{\bigskip \par \noindent}
\def\move-in{\parshape=1.75true in 5true in}
\theoremstyle{plain}
\newtheorem{lemma}{Lemma}
\newtheorem{propos}{Proposition}
\newtheorem{corol}{Corollary}
\newtheorem{defi}{Definition}

\newtheorem{rem}{Remark}
\newtheorem{nota}{Notation}
\newtheorem{alg}{Algorithm}
\newtheorem{thm}{Theorem}
\date{}
\begin{document}
\title[]{Computing symmetric rank for symmetric tensors.}
\author{Alessandra Bernardi, Alessandro Gimigliano, Monica Id\`a}
\address{CIRM--FBK, Trento, Italy} 
\email{bernardi@fbk.eu}, 
\address{Dip. di Matematica  and C.I.R.A.M., Univ.
degli Studi di Bologna,  Italy}
\email{gimiglia@dm.unibo.it}
\address{Dip. di Matematica, Univ. degli Studi di
Bologna, Italy}
\email{ida@dm.unibo.it}

\begin{abstract}
We consider the problem of determining the symmetric tensor rank for
symmetric tensors with an algebraic geometry approach. We give
algorithms for computing the symmetric rank for $2\times \cdots
\times 2$ tensors and for tensors of small border rank. From a
geometric point of view, we describe the symmetric rank strata for
some secant varieties of Veronese varieties.
\end{abstract}

\maketitle

\section{Introduction}

In this paper we study problems related to how to represent
symmetric tensors, this is a kind of question which is relevant in
many applications as in Electrical Engineering (Antenna Array
Processing \cite{ACCF}, \cite{DM} and Telecommunications
\cite{Ch}, \cite{dLC}); in Statistics (cumulant tensors, see
\cite{McC}), or in Data Analysis ( Independent Component Analysis
\cite{Co1}, \cite{JS}). For other applications see also
\cite{Co2}, \cite{CR}, \cite{LMV}, \cite{SBG}.

Let $t$ be a symmetric tensor $t\in\ S^{d}V$, where $V$ is an
$(n+1)$-dimensional vector space; the minimum integer $r$ such that
$t$ can be written as the sum of $r$ elements of the type
$v^{\otimes d}\in S^{d}V$ is called the {\it symmetric rank } of $t$
(Definition \ref{srk}).

 In most applications it turns out that the knowledge of the symmetric rank
 is quite useful, e.g. the symmetric rank of a symmetric tensor extends the
Singular Value Decomposition (SVD) problem for symmetric matrices
(see \cite{GVL}).

It is quite immediate to see that we can associate a homogeneous
polynomial in $K[x_0,...,x_n]_d$ to any symmetric tensor $t\in\
S^{d}V$  (see \ref{Sylvester algorithm}). It is a very classical
algebraic problem (inspired by a number theory problem posed by
Waring in 1770, see \cite{W}), to determine which is the minimum
integer $r$ such that a generic form of degree $d$ in $n+1$
variables can be written as a sum of $r$ $d$-th powers of linear
forms. This problem, known as the Big Waring Problem, is equivalent
to determining the symmetric rank of $t$.

If we regard  $\PP {{n+d\choose d}-1}$ as $\PPP (K[x_{0}, \ldots ,
x_{n}]_{d})$, then the Veronese variety $X_{n,d}\subset \PP
{{n+d\choose d}-1}$ is the variety that parameterizes those
polynomials that can be written as $d$-th powers of a linear form
(see Remark \ref{nud}). When we view  $\PP {{n+d\choose d}-1}$ as
 $\PPP (S^{d}V)$, where $V$ is an
$(n+1)$-dimensional vector space, the Veronese variety parameterizes
projective classes of symmetric tensors of the type $v^{\otimes
d}\in S^{d}V$ (see Definition \ref{veronese}).

 The set that
parameterizes tensors in $\PPP (S^{d}V)$ of a given symmetric rank
is not a closed variety. If we consider $\sigma_{r}(X_{n,d})$, the
$r$-th secant variety of $X_{n,d}$ (see Definition \ref{secant}),
this is the smallest variety containing all tensors of symmetric
rank $r$, and this for all $r$ up to the ``typical rank", i.e. the
first $r$ for which  $\sigma_{r}(X_{n,d})=\PPP (S^{d}V)$.
 The smallest $r$ such that $T\in
\sigma_{r}(X_{n,d})$ is called the {\it symmetric border rank} of
$T$ (Definition \ref{sbr}). This shows that, from a geometric point
of view, it seems more natural to study the symmetric
border rank of tensors rather than the symmetric rank.

A geometric formulation of the Waring problem for forms asks which
is the symmetric border rank of a generic symmetric tensor of
$S^{d}V$. This problem was completely solved by J. Alexander and A.
Hirschowitz who computed the dimensions of $\sigma_{r}(X_{n,d})$ for
any $r,n,d$ (see \cite{AH} for the original proof and \cite{BO}
for a recent proof).

Although the dimensions of the $\sigma_{r}(X_{n,d})$'s are now all
known, the same is not true for their defining equations: in general
for all $\sigma_{r}(X_{n,d})$'s  the equations coming from
catalecticant matrices (Definition \ref{catalecticant}) are known,
but in many cases they are not enough to describe their ideal; only
in a few cases our knowledge is complete (see for example \cite{K},
\cite{IK}, \cite{CGG}, \cite{Ot} and \cite{LO}). The knowledge
of equations which define $\sigma_{r}(X_{n,d})$, at least
set-theoretically, would give the possibility to compute the
symmetric border rank for any tensor in $S^{d}V$.

A first efficient method to compute the symmetric rank of a
symmetric tensor in $\PPP (S^{d}V)$ when $\dim(V)=2$ is due to
Sylvester \cite{Sy}. More than one version of that algorithm is
known (see \cite{Sy}, \cite{BCMT}, \cite{CS}). In Section
\ref{Two dimensional case} we present a new version of that
algorithm, which gives the symmetric rank of a tensor without
passing through an explicit decomposition of it. The advantage of
not giving an explicit decomposition is that this allows to much
improve the speed of the algorithm. Finding explicit decompositions
is a very interesting open problem (see also \cite{BCMT} and
\cite{LT} for a study of the case $\dim (V)\geq 2$).

The aim of this paper is to explore a ``projective geometry view" of
the problem of finding what are the possible symmetric ranks of a
tensor once its symmetric border rank is given. This amounts to
determining the symmetric rank strata of the varieties
$\sigma_r(X_{n,d})$. We do that in the following four cases:
$\sigma_{r}(X_{1,d})$ (for any $r$ and $d$, see also \cite{BCMT},
\cite{CS},  \cite{LT} and \cite{Sy});  $\sigma_2(X_{n,d})$,
$\sigma_3(X_{n,d})$ (any $n$,$d$, see Section 4);
$\sigma_r(X_{2,4})$, for $r\leq 5$. In the first three cases we
also give an algorithm to compute the symmetric rank. Some of these
results were known or partially known, with different approaches and
different algorithms, e.g in \cite{LT} bounds on the symmetric rank
are given for tensors in $\sigma_3(X_{n,d})$, while the possible
values of the symmetric rank on $\sigma_3(X_{2,3})$ can be found in
\cite{BCMT}, where an algorithm to find the decomposition is given.
In Section \ref{Two dimensional case} we also study the rank of
points on $\sigma_{2}(\Gamma_{d+1})\subset \PP d$,
 with respect to an elliptic normal curve $\Gamma_{d+1}$; for $d=3$,
$\Gamma_4$ gives another example (besides rational normal curves) of
a curve $C \subset \PP n$ for which there are points of $C$-rank
equal to $n$.


\section{Preliminaries}

We will always work with finite dimensional vector spaces defined
over an algebraically closed field $K$ of characteristic $0$.

\begin{defi}\label{srk}
\rm{ Let $V$ be a finite dimensional vector space.  The symmetric rank
$\sr(t)$ of a symmetric tensor $t\in S^{d}V$ is the minimum integer
$r$ such that there exist $v_{1}, \ldots , v_{r}\in V$ such  that
$t=\sum_{j=1}^{r} v_{j}^{\otimes d}$.}
\end{defi}

\begin{nota}\rm{ From now on we will indicate with $T$ the projective class
of a symmetric tensor $t\in S^{d}V$, i.e. if $t\in S^{d}V$ then
$T=[t]\in \PP {}(S^{d}V)$. We will write that an element $T\in \PP
{}(S^{d}V)$ has symmetric rank equal to $r$ meaning that there
exists a tensor $t\in S^{d}V$ such that $T=[t]$  and $\sr(t)=r$.}
\end{nota}

\begin{defi}\label{veronese}\rm{ Let $V$ be a vector space of dimension $n +1$.
The Veronese variety $X_{n,d}=\nu_d(\PP {} (V))
\subset \PP {}(S^{d}V)= \PP {{n+d\choose d}-1}$ is the variety given by the embedding
$\nu_d$ defined by the complete linear system of hypersurfaces of degree $d$ in $\PP n$}.
\end{defi}

\medskip
Veronese varieties parameterize projective classes of symmetric
tensors in $S^{d}V$ of symmetric rank $1$. Actually $T\in X_{n,d}$
if and only if there exists $v\in V$ such that $t=v^{\otimes d}$.

\begin{rem}\label{nud} \rm Let $V$ be a vector space of dimension $n$ and let $l\in V^{*}$ be a linear form.
Now define $\nu_{d}: \PP {}(V^{*}) \rightarrow \PP {}(S^{d}V^{*})$ as  $\nu_{d}([l])=[l^{d}]\in \PP {}(S^{d}V^{*})$.
The image of this map is indeed  the $d$-uple Veronese embedding of $\PP {}(V^{*})$.
\end{rem}

\begin{rem} \rm Remark \ref{nud} shows that, if $V$ is an $n$-dimensional vector space, then, given a basis for $V$, we can associate to any
 symmetric tensor $t\in S^{d}V$ of symmetric rank $r$  a
 homogeneous polynomial of degree $d$ in $n+1$ variables that can be written as a sum of $r$ $d$-th power
  of linear forms (see 3.1).
\end{rem}

\begin{nota}\rm{ If $v_{1}, \ldots , v_{s}$ belong to a vector space $V$, we will
denote with $<v_{1}, \ldots , v_{s}>$ the subspace spanned by them.
If $P_{1}, \ldots , P_{s}$ belong to a projective space $\PP n$ we
will use the same notation $<P_{1}, \ldots ,P_{s}>$ to denote the
projective subspace generated by them.}
\end{nota}

\begin{defi}\label{secant}\rm{Let $X\subset \PP N$ be a projective variety of dimension $n$.
We define the $s$-th secant variety of $X$ as follows:
$$\sigma_{s}(X):= \overline{ \bigcup_{P_{1}, \ldots , P_{s}\in  X} <P_{1}, \ldots , P_{s}>}.$$}
\end{defi}

\begin{nota}\label{sigma0} \rm{We will indicate with $\sigma_{s}^{0}(X)$ the set
  $ \bigcup_{P_{1}, \ldots , P_{s}\in X} <P_{1}, \ldots , P_{s}>$.}
\end{nota}

\begin{nota}\label{Grass}\rm{ With $G(k,V)$ we denote the Grassmannian of $k$-dimensional
subspaces of a vector space $V$, and with $\GG (k-1, \PP {}(V))$ we denote the $(k-1)$-dimensional
projective subspaces of the projective space $\PP {}(V)$. }
\end{nota}

\begin{rem}\label{brkremark1} \rm Let $X\subset \PP N$ be a non degenerate smooth variety. If
$P\in \sigma_{r}^{0}(X)\setminus \sigma_{r-1}^{0}(X)$ then the
minimum number of distinct points $P_1,...,P_s \in X$ such that  $P\in <P_1,...,P_s>$
 is obviously $r$, which is achieved on $ \sigma_{r}^{0}(X)$. We want to study what is that minimum number
 in $\sigma_{r}(X)\setminus( \sigma_{r}^{0}(X)\cup \sigma_{r-1}(X))$.
\end{rem}

\begin{propos}\label{brkremark2} Let $X\subset \PP N$ be a non degenerate smooth variety.
Let  $H_r$ be the irreducible component of the Hilbert scheme of
0-dimensional schemes of degree $r$ of $X$ containing $r$ distinct
points, and assume that for each $y\in H_r$, the corresponding
subscheme $Y$ of $X$ imposes independent conditions to linear forms.
Then for each $P\in \sigma_{r}(X)$ $\setminus \sigma_{r}^{0}(X)$
there exists a 0-dimensional scheme $Z\subset X$ of degree $r$  such
that $P\in <Z>\cong \PP {r-1}$.
\par
Conversely if there exists $Z\in H_{r}$ such that $P\in <Z>$, then $P\in \sigma_{r}(X)$.
\end{propos}

\begin{proof} Let us consider the map $\phi: H_r \to \Bbb G (r-1,\PP N)$, $\phi
(y)=<Y>$. The map $\phi$ is well defined since $\dim <Y>=r-1$ for
all $y\in H_r$ by assumption. Hence $ \phi (H_r)$ is closed in $\Bbb
G (r-1,\PP N)$.

\par Now let ${\mathcal I }\subset \PP N \times \Bbb G (r-1,\PP N)$ be the incidence variety,
 and $p$, $q$ its projections on $\PP N$ and on $ \Bbb G (r-1,\PP N)$
respectively; then, $A:= pq^{-1}  (\phi (H_r))$ is closed in $\PP N$. Moreover, $A$ is irreducible
since $H_r$ is irreducible, so $\sigma_{r}^{0}(X)$ is dense in $A$.
Hence $\sigma_{r}(X)=\overline {\sigma_{r}^{0}(X)}=A$.
\end{proof}

 In the following we will use Proposition \ref{brkremark2} when
$X=X_{n,d}$, a Veronese variety, in many cases.
\par \medskip
\begin{rem}\label{brkremark3}\rm Let $n=1$; in this case the Hilbert scheme of 0-dimensional
schemes of degree $r$ of $X =X_{1,d}$ is irreducible; moreover,
for all $y$ in the Hilbert scheme, $Y$ imposes independent
conditions to forms of any degree.

\par \medskip Also for $n=2$ the Hilbert scheme of 0-dimensional
schemes of degree $r$ of $X=X_{2,d}$ is irreducible.  Moreover, in
the cases that we will study here, $r$ is always small enough with
respect to $d$, so to imply that all the elements in the Hilbert
scheme impose independent conditions  to forms of degree $d$.

\par Hence in the two cases above $P\in \sigma_{r}(X)$ if and only if there exists a scheme
$Z\subset X$ of degree $r$  such that $P\in<Z>\simeq \PP {r-1}$.
\end{rem}

\par  \medskip Now we give an example which shows that
an $(r-1)$-dimensional linear space contained in $\sigma_{r}(X)$ is
not always spanned by a $0$-dimensional scheme of $X$ of degree $r$.
Let $n=2$, $d=6$, and consider $X=X_{2,6}=\nu_6(\PP 2)\subset \PP
{27}$ The first $r$ for which $ \sigma_{r}(X)$ is the whole of $\PP
{27}$ is 10; we will consider
 $ \sigma_{8}(X)\subset \PP
{27}$. Let $Z\in \PP 2$ be a scheme  which is the union of 8
distinct points on a line $L\subset \PP {2}$. The curve $\nu _6(L)$
is a rational normal curve $C_6$ in its span  $<C_6>\cong \PP 6$, so
$\dim <\nu_6(Z)>= 6$. Moreover, since $Z$ imposes only 6 conditions
to curves of degree six in $\PP 2$, then $\nu(Z)$ does not impose
independent conditions to linear forms in $\PP {27}$.  Now every
linear 7-dimensional space $\Pi \subset \PP {27}$ containing $C_6$,
meets $X$ along $C_6$ and no other point; hence there does not exist
a 0-dimensional scheme $B$ of degree $8$ on $X$ such that $ <B>
\supset <\nu_6(Z)>$ and $ <B>= \Pi$. On the other hand, consider a
1-dimensional flat family whose generic fiber $Y$ is the union of 8
distinct points on $X$ (hence $\dim <Y>= 7$) and such that $\nu(Z)$
is a special fiber of the family. If we consider the closure of the
corresponding family of linear spaces with generic fiber $<Y>$, this
is still is a 1-dimensional flat family, so it has to have a linear
space $\Pi_0 \cong \PP 7$ as special fiber. Hence  $\sigma _8(X)$
contains linear spaces of dimension 7 as $\Pi_0$, such that
 $<\nu_6(Z)> \subset \Pi_0$, but for no subscheme $Y'$ of degree 8 on $X$ we have $\Pi_0 = <Y'>$.

\begin{rem}  \rm A tensor $t\in S^{d}V$ with $\dim(V)=n+1$ has symmetric rank
$r$ if and only if  $T\in \sigma^{0}_{r}(X_{n,d})$ and, for any
$s<r$, we have that $T\notin \sigma^{0}_{s}(X_{n,d})$. In fact by
definition of symmetric rank of an element $T \in S^{d}V$, there
should exist $r$ elements (and no less) $T_{1}, \ldots , T_{r}\in
X_{n,d}$ corresponding to tensors $t_{1}, \ldots , t_{r}$ of
symmetric rank one such that $t=\sum_{i=1}^{r}t_{i}$. Hence $T\in
\sigma^{0}_{r}(X_{n,d})\setminus \sigma_{r-1}^{0}(X_{n,d})$.
\end{rem}

\begin{defi}\label{sbr} \rm{ If  $T\in \sigma_{s}(X_{n,d})\setminus \sigma_{s-1}(X_{n,d})$, we say that
$t$ has symmetric border rank $s$, and we write $\underline{\sr}(t)=s$.}
\end{defi}

\begin{rem} \rm The symmetric border rank of $t \in S^{d}V$, with $\dim(V)=n+1$,  is the smallest
$s$  such that $T\in \sigma_{s}(X_{n,d})$. Therefore  $\sr(t) \geq
\underline{\sr}(t)$. Moreover if $T\in \sigma_{s}(X_{n,d})\setminus
\sigma_{s}^{0}(X_{n,d})$ then $\sr(t)>s$.
\end{rem}

The following notation will turn out to be useful in the sequel.

\begin{nota}\label{sigmasq}\rm{ We will indicate with $\sigma_{b,r}( X_{n,d})\subset \PP {}(S^{d}V)$ the set:
$$\sigma_{b,r}( X_{n,d}):= \{T \in \sigma_{b}(X_{n,d})\setminus \sigma_{b-1}(X_{n,d})| \sr (T) =r\},$$
i.e. the set of the points in $\PPP (S^{d}V)$ corresponding to
symmetric tensor whose symmetric border rank is $b$ and whose
symmetric rank is $r$.}
\end{nota}
\medskip
It is not easy to get a geometric description of the loci
$\sigma_{b,r}( X_{n,d})$'s; we think that (when the base field is
algebrically closed) they should be locally closed (when $n=1$, i.e.
for rational normal curves, this follows from Corollary
\ref{strata}), but we have no general reference for that.

\bigskip


\section{Two dimensional case}\label{Two dimensional case}

In this section we will restrict to the case of a $2$-dimensional
vector space $V$. We first describe the Sylvester algorithm which
gives the symmetric rank of a symmetric tensor $t\in S^{d}V$ and a
decomposition of $t$ as a sum of $r=\sr(t)$ symmetric tensors of
symmetric rank one (see \cite{Sy} \cite{CS}, \cite{BCMT}). Then
we give a geometric description of the situation and a slightly
different algorithm which produces the symmetric rank of a symmetric
tensor in $S^{d}V$ without giving explicitly its decomposition. This
algorithm makes use of a result (see Theorem \ref{curve}) which
describes the rank of tensors on the secant varieties of the rational
normal curve $C_d = X_{1,d}$; this Theorem has been proved in the
unpublished paper \cite{CS} (see also \cite{LT}); here we give a
proof which uses only classical projective geometry.

Moreover we extend part of that result to elliptic normal curves, see
Theorem \ref{ellcurve}.


\subsection{The Sylvester algorithm}\label{Sylvester algorithm}

Let $p\in K[x_{0}, x_{1}]_{d}$ be a homogeneous polynomial of degree
$d$ in two variables: $p(x_{0},x_{1})=\sum_{k=0}^{d}
a_{k}x_{0}^{k}x_{1}^{d-k}$; then we can associate to the form $p$ a
symmetric tensor $t\in S^{d}V\simeq K[x_{0}, x_{1}]_{d}$ where $t=(b_{i_{1}, \ldots , i_{d}})_{i_{j}\in\{0,1\}; j=1, \ldots ,
d}$, and
$ b_{i_{1}, \ldots , i_{d}}={d\choose k}^{-1}\cdot a_{k}$ for any $d$-uple
$(i_{1}, \ldots , i_{d})$ containing exactly $k$ zeros. This
correspondence is clearly one to one:
\begin{equation}\label{corrisp}\begin{array}{rcl}K[x_{0}, x_{1}]_{d} &\leftrightarrow &S^{d}V
\\
\sum_{k=0}^{d} a_{k}x_{0}^{k}x_{1}^{d-k} &\leftrightarrow &(b_{i_{1}, \ldots , i_{d}})_{i_{j}=0,1;\; j=1, \ldots , d}
\end{array}\end{equation}
with $(b_{i_{1}, \ldots , i_{d}})$ as above.

\smallskip
The algorithm uses Catalecticant matrices, which are matrices that
we can associate to a polynomial $p(x_{0},x_{1})=\sum_{k=0}^{d}
a_{k}x_{0}^{k}x_{1}^{d-k}$, or to the symmetric tensor $t$
associated to it. We give below the definition of Catalecticant
matrices $M_{d-r,r}^n$ and $M_{d-r,r}(t)$ in the general case, see \cite{Ge} or \cite{K}; $M_{d-r,r}(t)$ is
also called Hankel matrix in \cite{BCMT}).

\begin{defi}\label{catalecticant}
Let $R=k[x_0,...,x_n]$ and $i,j,d \in \mathbb N$ with $i+j=d$.
Consider the bilinear map given by multiplication:
$$R_i\times R_j \rightarrow R_d.$$
If we fix in $R_i$,$R_j$ the natural bases given by monomials (say
in lex order), the map above can be represented by a $({n+i \choose
n})\times ({n+j \choose n})$ matrix $A$. The $(i,j)$-catalecticant
Matrix of $R$: $M_{i,j}^n$ is the $({n+i \choose n})\times ({n+j
\choose n})$ matrix whose entries are the indeterminates
$z_{\underline{\alpha}}$,
$\underline{\alpha}=(\alpha_0,...,\alpha_n) \in \mathbb {N }^n$, with $|\underline{\alpha}|=d$. For each entry $m_{u,v}$ of $M_{i,j}^n$, we have
 $m_{u,v}= z_{\underline{\alpha}}$ if the entry $a_{u,v}$
in $A$ is associated to multiplication of two monomials which yields
$\underline{x}^{\underline{\alpha}} =
x_0^{\alpha_0}...x_n^{\alpha_n}$.

\medskip
Example. Let $n=d=2$, $i=j=1$; we get:
$$M_{1,1}^2 = \begin{pmatrix}z_{2,0,0}&  z_{1,1,0}& z_{1,0,1}\cr z_{1,1,0}& z_{0,2,0}&z_{0,1,1}\cr
z_{1,0,1}&z_{0,1,1}&z_{0,0,2}
\end{pmatrix}.$$

If we consider the new variables as coordinates in $\PP N$, $N =
{n+d \choose n}-1$, it is well known that the ideal of the $2\times
2$ minors of $M_{i,j}^n$ is the defining ideal of the Veronese
variety $X_{n,d}=\nu_d(\PP n)$.

Now consider a form $p\in R_d$. The $(i,j)$-catalecticant Matrix of
$p$, $M_{i,j}(p)$, is the numerical matrix which yields:
$$ (x_0^{\beta_0},...,x_n^{\beta_n})\cdot M_{i,j}(p)\cdot {}^t(x_0^{\gamma_0},...,x_n^{\gamma_n}) = f(x_0,...,x_n),$$
where  $ \{x_0^{\beta_0},...,x_n^{\beta_n}\}$,
$\{x_0^{\gamma_0},...,x_n^{\gamma_n}\}$ are the bases for
$R_i$,$R_j$, respectively.

Since we are more interested in tensors, we will always write
$M_{i,j}(t)$ or $M_{i,j}(T)$, instead of $M_{i,j}(p)$, where $t$ is
the symmetric tensor associated to $p$ (as we did at the beginning
of the section in the 2-dimensional case), and $T$ is its projective
class in $\PPP (S^d(V))$.
\end{defi}

\medskip
\begin{rem}\label{catalecticant1}
\rm{ When considering the two dimensional case, it is easier to
describe $M_{i,j}(p)$ more explicitely. Let
$p(x_{0},x_{1})=\sum_{k=0}^{d} a_{k}x_{0}^{k}x_{1}^{d-k}$, and $t
=(b_{i_1},...,b_{i_d})_{i_{j}=0,1;\; j=1, \ldots , d}\in S^{d}V$ be
the symmetric tensor associated to $p$, as we did at the beginning
of section. Then the Catalecticant matrix $M_{d-r,r}(t)$ associated
to $t$ (or to $p$) is the $(d-r+1) \times (r+1)$ matrix with
entries:\quad
 $c_{i,j}={d\choose
i}^{-1}a_{i+j-2}$ with $i=1, \ldots , d-r$ and $j=1, \ldots , r$. }
\end{rem}

We describe here the version of the Sylvester algorithm that can be
found in \cite{Sy}, \cite{CS}, or \cite{BCMT}:
\begin{alg}\label{SylvesterOld}\rm{
\textbf{Input}: A binary form $p(x_{0},x_{1})$ of degree $d$ or, equivalently, its associated symmetric tensor $t$.\\
\textbf{Output}: A decomposition of $p$ as
$p(x_{0},x_{1})=\sum_{j=1}^{r}\lambda_{j}l_{j}(x_{0}, x_{1})^{d}$
with $\lambda_{j}\in K$ and $l_{j}\in K[x_{0},x_{1}]_{1}$ for $j=1,
\ldots ,r$ with $r$ minimal.
\begin{enumerate}
\item Initialize $r=0$;
\item\label{syl2} Increment $r \leftarrow r+1$;
\item If the rank of the matrix $M_{d-r,r}(t)$ is maximum, then go to step $2$;
\item Else compute a basis $\{l_{1}, \ldots , l_{h}\}$ of the right kernel of $M_{d-r,r}$;
\item Specialization:
\begin{itemize}
\item Take a vector $q$ in the right kernel of $M_{d-r,r}(t)$, e.g. $q=\sum_{i}\mu_{i}l_{i}$;
\item Compute the roots of the associated polynomial $q(x_{0},
x_{1})=\sum_{h=0}^{r}q_{h}x_{0}^{h}x_{1}^{d-h}$ and denote them by
$(\alpha_{j},\beta_{j})$, where
$|\alpha_{j}|^{2}+|\beta_{j}|^{2}=1$;
\item If the roots are not distinct in $\PP 1$, go to step \ref{syl2};
\item Else if $q(x_{0},x_{1})$ admits $r$ distinct roots then compute coefficients $\lambda_{j}$, $1\leq j \leq r$, by solving the linear system below:
$$\left( \begin{array}{ccc}
\alpha_{1}^{d} &\cdots & \alpha_{r}^{d}\\
\alpha_{1}^{d-1}\beta_{1} &\cdots & \alpha_{r}^{d-1}\beta_{r}\\
\alpha_{1}^{d-2}\beta_{1}^{2} &\cdots & \alpha_{r}^{d-2}\beta_{r}^{2}\\
\vdots &\vdots &\vdots \\
\beta_{1}^{d} &\cdots &\beta_{r}^{d}
\end{array}
\right)
\lambda =
\left(
\begin{array}{c}
a_{0}\\
1/d a_{1}\\
{d\choose 2}^{-1}a_{2}\\
\vdots \\
a_{d}
\end{array}
\right);$$
\end{itemize}
\item The required decomposition is $p(x_{0},x_{1})=\sum_{j=1}^{r}\lambda_{j}l_{j}(x_{0},x_{1})^{d}$,
where $l_{j}(x_{0},x_{1})=(\alpha_{j}x_{1}+\beta_{j}x_{2})$.
\end{enumerate}}
\end{alg}


\subsection{Geometric description}

If $V$ is a two dimensional vector space, there is a well known
isomorphism between $\bigwedge^{d-r+1}(S^{d}V)$ and $
S^{d-r+1}(S^{r}V)$ ,(see \cite{Mu}). When $d\geq r$ such isomorphism can be
interpreted in terms of projective algebraic varieties; it allows to
view the  $(d-r+1)$-uple Veronese embedding of $\PP r$,  as the set
of $(r-1)$-dimensional projective subspaces of $\PP {d}$ that are
$r$-secant to the rational normal curve. The description of this
result, via coordinates, was originally given by A. Iarrobino, V.
Kanev (see \cite{IK}). We give here the description appeared in
\cite{AB} (Lemma 2.1) (Notations as in \ref{Grass}).

\begin{lemma}\label{LemmaAB} \label{nu_d(Pn)intG(n-1,n+d-1)} Consider the map
$\phi_{r,d-r+1}:\PP{}(K[t_0,t_1]_r)\to G(d-r+1,K[t_0,t_1]_{d})$ that
sends the class of $p_0\in K[t_0,t_1]_r$ to the
$(d-r+1)$-dimensional subspace of $K[t_0,t_1]_{d}$ of forms of the
type $p_0q$, with $q\in K[t_0,t_1]_{d-r}$. Then the following hold:
\item{(i)} The image of $\phi_{r,d-r+1}$, after the Pl\"ucker
embedding of $G(d-r+1,K[t_0,t_1]_{d})$, is the Veronese variety $X_{r,d-r+1}$.
\item{(ii)} Identifying, by duality, $G(d-r+1,K[t_0,t_1]_{d})$ with the
Grassmann variety of subspaces of dimension $r-1$ in
$\PP{}(K[t_0,t_1]_{d}^*)$, the above Veronese variety is the
set of $r$-secant spaces to a rational normal curve
$C_{d}\subset\PP{}(K[t_0,t_1]_{d}^*)$.
\end{lemma}

\begin{proof} Write
$p_0=u_0t_0^r+u_1t_0^{r-1}t_1+\dots+u_rt_1^r$. Then a basis of
the subspace of $K[t_0,t_1]_{d}$ of forms of the
type $p_0q$ is given by:
\begin{equation}\label{rsecante}
\begin{array}{l}
u_0t_0^{d}+\cdots +u_rt_0^{d-r}t_1^r \\
\;\; \; u_0t_0^{d-1}t_1 +\cdots +u_rt_0^{d-r-1}t_1^{r+1} \\
 \;\; \;\; \; \ddots\\
\; \; \; \; \; \; \; \; \; u_0t_0^rt_1^{d-r}+\cdots +u_rt_1^{d}.
\\
\end{array}
 \end{equation}
The coordinates of these elements with respect to the basis
$\{t_0^{d},t_0^{d-1}t_1,\dots,t_1^{d}\}$ of
$K[t_0,t_1]_{d}$ are thus given by the rows of the matrix
$$\left(\begin{array}{cccccccc}
u_0&u_1&\dots&u_r&0&\dots&0&0\\
0&u_0&u_1&\dots&u_r&0&\dots&0\\
\vdots&\ddots&\ddots&\ddots&&\ddots&\ddots&\vdots\\
0&\dots&0&u_0&u_1&\dots&u_r&0\\
0&\dots&0&0&u_0&\dots&u_{r-1}&u_r
\end{array}
\right).$$
The standard Pl\"ucker coordinates of the subspace
$\phi_{r,d-r+1}([p_0])$ are the maximal minors of this matrix. It is
known (see for example \cite{AP}), that these minors form a basis of
$K[u_0,\dots,u_r]_{d-r+1}$, so that the image of $\phi$ is indeed a
Veronese variety, which proves (i).

To prove (ii), we recall some standard facts from \cite{AP}. Take
homogeneous coordinates $z_0,\dots,z_{d}$ in
$\PP{}(K[t_0,t_1]_{d}^*)$ corresponding to the dual basis of
$\{t_0^{d},t_0^{d-1}t_1,\dots,t_1^{d}\}$. Consider
$C_{d}\subset\PP{}(K[t_0,t_1]_{d}^*)$ the standard rational normal
curve with respect to these coordinates. Then, the image of $[p_0]$
by $\phi_{r,d-r+1}$ is precisely the $r$-secant space to $C_{d}$
spanned by the divisor on $C_{d}$ induced by the zeros of $p_0$.
This completes the proof of (ii).
\end{proof}

 Since $\dim(V)=2$,  the Veronese variety of $\PP {}(S^{d}V)$ is  the rational normal curve
 $C_{d} \subset \PP d$. Hence, a symmetric tensor $t\in S^{d}V$ has symmetric rank $r$ if
 and only if $r$ is the minimum integer for which there exists a $\PP {r-1}=\PP {}(W)\subset \PP {}(S^{d}V)$
 such that $T\in \PP {}(W)$ and $\PP {}(W)$ is $r$-secant to the rational normal curve  $C_{d}\subset\PP {}(S^{d}V)$ in $r$ distinct points.
\\
Consider the maps:
\begin{equation}\label{star}\PP {}(K[t_{0},t_{1}]_{r}) \stackrel{\phi_{r, d-r+1}}{\rightarrow}\GG (d-r, \PP {}( K[t_{0},t_{1}]_{d}))
\stackrel{\alpha_{r, d-r+1}}{\simeq} \GG (r-1,  \PP {}(K[t_{0},t_{1}]_{d})^{*}).\end{equation}
Clearly, since $\dim(V)=2$, we can identify $\PP {}(K[t_{0},t_{1}]_{d})^{*} $ with $\PP {}(S^{d}V)$, hence the
Grassmannian $\GG (r-1 , \PP {}(K[t_{0},t_{1}]_{d})^{*})$ can be identified with $\GG (r-1, \PP {}(S^{d}V))$. \\
Now, by Lemma \ref{LemmaAB}, a projective subspace $\PP {}(W)$ of
$\PP {}(K[t_{0},t_{1}]_{d})^{*}\simeq \PP {}(S^{d}V)\simeq \PP d$ is
$r$-secant to $C_{d}\subset \PP {}(S^{d}V)$ in $r$ distinct points
if and only if it belongs to  $\mathrm{Im}(\alpha_{r, d-r+1} \circ
\phi_{r,d-r+1})$ and the preimage of $\PP {}(W)$ via $\alpha_{r,
d-r+1} \circ \phi_{r,d-r+1}$ is a polynomial with $r$ distinct
roots.
\\
Therefore, a symmetric tensor $t\in S^{d}V$ has symmetric rank $r$ if and only if $r$ is the minimum integer for which:
\begin{enumerate}
\item $T$ belongs to an element  $\PP {}(W)\in \mathrm{Im}(\alpha_{r, d-r+1} \circ \phi_{r,d-r+1} )\subset \GG (r-1, \PP {}(S^{d}V))$,
\item there exists a polynomial $p_{0}\in K[t_{0},t_{1}]_{r}$ such that
$\alpha_{r, d-r+1} ( \phi_{r,d-r+1}([p_{0}]) )=\PP {}(W)$ and $p_{0}$  has $r$ distinct roots.
\end{enumerate}
Fix the natural basis $\Sigma=\{t_{0}^{d}, t_{0}^{d-1}t_{1}, \ldots
, t_{1}^{d}\}$ in $K[t_{0}, t_{1}]_{d}$. Let $\PP {}(U)$ be a
$(d-r)$-dimensional projective subspace of $\PP {}(K[t_{0},
t_{1}]_{d})$. The proof of Lemma \ref{LemmaAB} shows that $\PP
{}(U)$ belongs to the image of $\phi_{r,d-r+1}$ if and only if there
exist $u_{0}, \ldots , u_{r}\in K$ such that $U=<p_{1}, \ldots ,
p_{d-r+1}>$ with $p_{1}=(u_{0}, u_{1},\ldots , u_{r}, 0, \ldots ,
0)_{\Sigma}$, $p_{2}=(0, u_{0}, u_{1},\ldots , u_{r}, 0, \ldots ,
0)_{\Sigma}$, . . . , $ p_{d-r+1}=(0, \ldots , 0, u_{0},
u_{1},\ldots , u_{r})_{\Sigma}$.
\\
Now let $\Sigma^{*}=\{z_{0}, \ldots ,  z_{d}\}$ be the dual basis of
$\Sigma$. Therefore there exists a $W\subset S^{d}V$ such that  $\PP
{}(W)=\alpha_{r,d-r+1}(\PP {}(U))$ if and only if  $W=H_{1}\cap
\cdots \cap H_{d-r+1}$ and the $H_{i}$'s are as follows:
$$
\begin{array}{rl}
H_{1}:& u_0z_{0}+\cdots +u_rz_{r}=0 \\
H_{2}: & \;\; \; u_0z_{1} +\cdots +u_rz_{r+1}=0 \\
\;\; \;\; \;    \;\; \;&\;    \;\;\; \;\; \; \ddots\\
H_{d-r+1}: & \; \; \; \; \; \;  u_0z_{d-r}+\cdots +u_rz_{d}=0.
\end{array}$$
This is sufficient to conclude that $T\in \PP {}(S^{d}V)$ belongs
to an $(r-1)$-dimensional projective subspace of $\PP {}(S^{d}V)$
that is in the image of $\alpha_{r, d-r+1} \circ \phi_{r,d-r+1}$
defined in (\ref{star}) if and only if there exist $H_{1}, \ldots
, H_{d-r+1}$ hyperplanes in $S^{d}V$ as above such that $T\in
H_{1}\cap \ldots \cap H_{d-r+1}$.
\\
Given $t=(a_{0}, \ldots ,a_{d})_{\Sigma^{*}}\in S^{d}V$, $T\in
H_{1}\cap \ldots \cap H_{d-r+1}$ if and only if the following linear
system admits a non trivial solution:
$$\left\{ \begin{array}{l}
u_0a_{0}+\cdots +u_ra_{r}=0 \\
u_0a_{1} +\cdots +u_ra_{r+1}=0 \\
 \vdots\\
u_0a_{d-r}+\cdots +u_ra_{d}=0.
\end{array}\right.$$
If $d-r+1<r+1$ this system admits an infinite number of solutions.\\
If $r\leq d/2$, it admits a non trivial solution if and only if  all
the maximal $(r+1)$-minors of the following $(d-r+1)\times (r+1)$
catalecticant matrix, defined in Definition  \ref{catalecticant},
vanish :
$$M_{d-r,r}= \left(
\begin{array}{ccc}
a_{0} & \cdots & a_{r}\\
a_{1} & \cdots & a_{r+1}\\
\vdots && \vdots \\
a_{d-r} & \cdots & a_{d}
\end{array}
\right).$$

The following three remarks contain results on rational normal
curves and their secant varieties that are classically known and
that we will need in our description.

\begin{rem} \rm The dimension of $\sigma_{r}(C_{d})$ is the minimum between $2r-1$
and $d$. Actually $\sigma_{r}(C_{d})\subsetneq \PP d$ if and only if $1\leq r < \left\lceil \frac{d+1}{2} \right\rceil$.
\end{rem}

\begin{rem} \rm An element $T\in \PP d$ belongs to $\sigma_{r}(C_{d})$ for
$1\leq r < \left\lceil \frac{d+1}{2} \right\rceil$  if and only if
the catalecticant matrix $M_{r,d-r}$ defined in Definition
\ref{catalecticant} does not have maximal rank.
\end{rem}

\begin{rem} \rm Any divisor $D\subset C_d$, with $\deg D\leq d+1$,  is such that $\dim <D> = \deg D
-1$.
\end{rem}

The following result has been proved by G. Comas and M. Seiguer in
the unpublished paper \cite{CS} (see also \cite{LT}), and it
describes the structure of the stratification by symmetric rank of
symmetric tensors in $S^{d}V$ with $\dim(V)=2$. The proof we give
here is a strictly ``projective geometry" one.
\bigskip

\begin{thm}\label{curve}
 Let $V$ be a 2-dimensional vector space and $X_{1,d} = C_d\subset \PPP (S^dV)$, be the rational normal curve, parameterizing
 decomposable symmetric tensors  ($C_d = \{T \in \PPP (S^{d}V) \, |\, \sr (T) =1\}$), i.e. homogeneous
 polynomials in $K[t_0,t_1]_d$ which are $d$-th powers of linear forms. Then:
$$
 \forall \, r,\ 2\leq r\leq \left\lceil{d+1\over 2}\right\rceil: \quad \qquad \sigma_r(C_d) \setminus  \sigma_{r-1}(C_d) = \sigma_{r,r}(C_{d})\cup \sigma_{r, d-r+2}(C_{d})$$
where $\sigma_{r,r}(C_{d})$ and $ \sigma_{r, d-r+2}(C_{d})$ are defined in Notation \ref{sigmasq}.
\end{thm}

\begin{proof}
\medskip
Of course, for all $t\in S^{d}V$, if $\sr (t)=r$, with $r \leq
\lceil{d+1\over 2}\rceil$, we have $T\in \sigma_r(C_d)
\setminus\sigma_{r-1}(C_d)$. Thus we have to consider the case $\sr
(t) > \lceil{d+1\over 2}\rceil$, which can happen only if $T \in \sigma_r(C_d) \setminus  \sigma_{r-1}(C_d)$ and $\sr
(t) >r$, i.e. $T\notin \sigma_r^0(C_d)$.

If a point in $K[t_0,t_1]_d^*$ represents a tensor $t$ with $\sr
(t)>\lceil{d+1\over 2}\rceil$, then we want to show that  $\sr (t) =
d-r+2$, where $r$ is the minimum integer such that $T\in
\sigma_r(C_d)$, $r\leq  \lceil{d+1\over 2}\rceil$.

First let us consider the case $r=2$. Let $T \in
\sigma_2(C_d)\setminus C_d$. If $\sr (t) >2$, then $T$ lies on a
line $t_P$, tangent to $C_d$ at a point $P$ (this is because $T$ has
to lie on a $\PP 1$ which is the image of a non-reduced form of
degree 2: $p_0=l^2$ with $l\in K[x_{0},x_{1}]_{1}$, otherwise $\sr
(t)=2$). We want to show that $\sr (t)=d$. If $\sr (t)= r < d$,
there would exist distinct points $P_1,\ldots, P_{d-1}\in C_d$, such
that $T\in <P_1,\ldots,P_{d-1}>$; in this case the hyperplane $H
=<P_1,\ldots ,P_{d-1},P>$ would be such that $t_P \subset H$, but
this is a contradiction, since $H \cap C_d = 2P+P_1+\cdots+P_{d-1}$
has degree $d+1$.

Notice that  $\sr (t) = d$ is possible, since obviously  there is a
$(d-1)$-space (i.e. a hyperplane) through $T$ cutting $d$ distinct
points on $C_d$ (any generic hyperplane through $T$ will do). This
also shows that $d$ is the maximum possible rank.

Now let us generalize the procedure above; let $T \in
\sigma_r(C_d)\setminus\sigma_{r-1}(C_d)$,  $r \leq \lceil{d+1\over
2}\rceil$; we want to prove that if $\sr (t) \neq r$, then $\sr (t)
= d-r+2$. Since $\sr (t)>r$, we know that $T$ must lie on a $\PP
{r-1}$ which cuts a non-reduced divisor $Z\in C_d$ with $\deg
(Z)=r$; therefore there is a point $P\in C_d$ such that $2P \in Z$.
If we had $\sr (t)\leq d-r+1$, then $T$ would be on a $\PP {d-r}$
which cuts $C_d$ in distinct points $P_1,\ldots,P_{d-r+1}$; if that
were true the space $<P_1,\ldots,P_{d-r+1}, Z-P>$ would be $(d-1 -
\deg (Z-2P)\cap \{P_1,\ldots,P_{d-r+1}\})$-dimensional and cut
$P_1+\cdots+P_{d-r+1} + Z - (Z-2P)\cap \{P_1,\ldots,P_{d-r+1}\}$ on $C_d$,
which is impossible.

So we got $\sr (t) \geq d-r+2$; now we have to show that the rank is
 actually $d-r+2$. Let's consider the divisor $Z-2P$ on $C_d$; we
have $\deg (Z-2P) = r-2$, and the space $\Gamma = <Z-2P,T>$ which is
$(r-2)$-dimensional since $<Z-2P>$ does not contain $T$ (otherwise
$T\in \sigma _{r-2}(C_d)$). We will be finished if we show that the
generic divisor of the linear series cut on $C_d$ by the hyperplanes
containing $\Gamma $ is reduced.

If it is not, there should be a fixed non-reduced part of the
series, i.e. there should exist at least a fixed divisor of type
$2Q$. If this is the case, each hyperplane through $\Gamma$ would
contain $2Q$, hence $2Q \subset \Gamma$, which is impossible, since
we would have $\deg (\Gamma \cap C_d) = r$, while $\dim \Gamma =
r-2$.

Thus $\sr (t) =  d-r+2$, as required.
\end{proof}

\begin{rem}\rm (Rank for monomials) In the proof above we have used the fact that (see Proposition 11) if $t$ is a symmetric
tensor such that $T\in \sigma_r(C_d)\setminus \sigma_{r-1}(C_d)$,
and $T\notin \sigma_r^0(C_d)$, then there exists a non reduced
$0$-dimensional scheme $Z\subset \PP d$, which is a divisor of
degree $r$ on $C_d$, such that $T\in <Z>$. Let $Z= m_1P_1+\dots
+m_sP_s$, with $P_1,\dots, P_s$ distinct points on the curve,
$m_1+\dots +m_s=r$ and $m_i\geq 2$ for at least one value of $i$.
Then $t^*$ can be written as
$$ t^*= l_1^{d-m_1+1}f_1+\dots + l_s^{d-m_s+1}f_s$$
where $l_1, \dots, l_s$ are homogeneous linear forms in two
variables and each $f_i$ is a homogeneous form of degree $m_i-1$ for
$i=1,\dots,s$.

In the theorem above it is implicitly proved that each form of this
type has symmetric rank $d-r+2$.
\medskip
In particular, every monomial of type $x^{d-s}y^{s}$ is such that
$$\sr (x^{d-s}y^{s}) = \max \{d-s+1 , s+1\}.$$
\end{rem}

\begin{nota}\label{tau} For all smooth projective varieties
$X,Y \subset \PP d$, we denote with $\tau(X)$  the $tangential$
$variety$ to $X$, i.e. the closure of the union of all its
projective embedded tangent spaces at its points, and with $J(X,Y)$,
the $join$ of $X$ and $Y$, i.e. the closure of the union of all the
lines $<x,y>$, for $x \in X$ and $y\in Y$.
\end{nota}

From the proof of Theorem \ref{curve}, we can also deduce the following result
which describes the strata of high rank on each $\sigma _r (C_d)$:

\begin{corol}\label{strata} Let $C_d \subset \PP d$, $d>2$; then we have:
\begin{itemize}
\item $\sigma_{2,d}(C_d) = \tau (C_d) \setminus C_d$;
\item For all $r$, with\quad  $3\leq r < {\frac{d+2}{2}}: \quad  \qquad
\sigma_{r,d-r+2}(C_{d})= J(\tau (C_d),\sigma_{r-2}(C_d)) \setminus  \sigma_{r-1}(C_{d})$.
\end{itemize}
\end{corol}


\subsection{A result on elliptic normal curves.}

We can use the same kind of construction we used for rational normal curves to prove the following
result on elliptic normal curves.

\begin{nota} \rm{ If $\Gamma_{d+1}\subset \PP d$, with $d\geq 3$, is an elliptic normal curve, and $T\in \PP d$,
we say that $T$ has rank $r$ with respect to $\Gamma_{d+1}$ and we
write $r = \mathrm{rk}_{\Gamma_{d+1}}(T)$, if $r$ is the minimum number of
points of $\Gamma_{d+1}$ such that $T$ depends linearly on them.
\\
In the following the $\sigma_{i,j}( \Gamma_{d+1})$'s are defined as
in Notation \ref{sigmasq}, but with respect to $\Gamma_{d+1}$, i.e.
$\sigma_{i,j}(\Gamma_{d+1}) = \{ T \in \PP d | \mathrm{rk}_{\Gamma
_{d+1}}(t) = j, T\in \sigma_i(\Gamma_{d+1})\}$.}
\end{nota}

\begin{thm}\label{ellcurve}
 Let $\Gamma_{d+1}\subset \PP {d}$, $d\geq 3$, be an elliptic normal curve, then:
\begin{itemize}
\item $When\ d= 3,\  we\ have: \qquad \sigma_2 (\Gamma_{4}) \setminus  \Gamma_{4} =
\sigma_{2,2}( \Gamma_{4})\cup \sigma_{2,3}(\Gamma_4);
 (\ here\ \sigma_2 (\Gamma_{4}) = \PP 3).$
\item For $d\geq 4$: $\quad \qquad \sigma_2 (\Gamma_{d+1}) \setminus  \Gamma_{d+1} = \sigma_{2,2}( \Gamma_{d+1})\cup \sigma_{2, d-1}( \Gamma_{d+1}).$
\end{itemize}
\noindent Moreover $\sigma_{2,3}(\Gamma_4) = \{T\in \tau(\Gamma_4)\ | \ two \ tangent\ lines\ to\ \Gamma_4\ meet \ in\ T \}$.
\end{thm}

\begin{proof}
First let $d\geq 4$; let $T \in \sigma_2(\Gamma
_{d+1})\setminus\Gamma_{d+1}$. If $\mathrm{rk}_{\Gamma_{d+1}} (T) >2$, it
means that $T$ lies on a line $t_P$, tangent to $\Gamma_{d+1}$ at a
point $P$. We want to show that $\mathrm{rk}_{\Gamma_{d+1}} (T)=d-1$. First
let us check that we cannot have $\mathrm{rk}_{\Gamma_{d+1}} (T)= r < d-1$.
In fact, if that were the case, there would exist points
$P_1,\ldots,P_{d-2}\in \Gamma_{d+1}$, such that $T\in
<P_1,\ldots,P_{d-2}>$; in this case the space
$<P_1,\ldots,P_{d-2},P>$ would be $(d-2)$-dimensional, and such that
$<P_1,\ldots ,P_{d-2},2P> = <P_1,\ldots,P_{d-2},P>$, since $T$ is on
$<P_1,\ldots , P_{d-2}>$, so the line $<2P> = t_P$ is in
$<P_1,\ldots ,P_{d-2},P>$ already. But this is a contradiction,
since $<P_1,\ldots ,P_{d-2},2P>$ has to be $(d-1)$-dimensional (on
$\Gamma_{d+1}$ every divisor of degree at most $d$ imposes independent
conditions to hyperplanes).

Now we want to check that $\mathrm{rk}_{\Gamma_{d+1}} (T) \leq d-1$.  We
have to show that there exist $d-1$ distinct points
$P_1,\ldots,P_{d-1}$ on $\Gamma_{d+1}$, such that $T \in
<P_1,\ldots,P_{d-1}>$.  Consider the hyperplanes in $\PP d$
containing  the line $t_P$; they cut a $g^{d-2}_{d+1}$ on
$\Gamma_{d+1}$, which is made of the fixed divisor $2P$, plus a
complete linear series $g^{d-2}_{d-1}$, which is of course very
ample; among the divisors of this linear series, the ones which span
a $\PP {d-2}$ containing $T$ form a sub-series $g^{d-3}_{d-1}$,
whose generic element is smooth (this is always true for a subseries
of codimension one of a very ample linear series), hence it is made
of $d-1$ distinct points whose span contains $T$, as required.

Now let $d=3$; obviously $\sigma_2({\Gamma_4}) = \PP 3$; if we have
a point $T\in (\sigma_2({\Gamma_4})\setminus {\Gamma_4})$, then $T$
is on a tangent line $t_P$ of the curve. Consider the planes through
$t_P$; they cut a $g^1_2$ on $\Gamma_4$ outside $2P$; each divisor
$D$ of such $g^1_2$ spans a line which meets $t_P$ in a point
($<D>+<2P>$ is a plane in $\PP 3$), so the $g^1_2$ defines a $2:1$
map $\Gamma_4 \rightarrow t_P$ which, by Hurwitz theorem, has four
ramification points. Hence for a generic point of $t_P$ there is a
secant line through it (i.e. it lies on $\sigma_{2,2}(\Gamma_4)$),
but for those special points no such line exists (namely, for the
points in which two tangent lines at $\Gamma_4$ meet), hence those
points have $\mathrm{rk}_{\Gamma_4}=3$ (a generic hyperplane through one
point cuts 4 distinct points on $\Gamma_4$, and three of them span
it).
\end{proof}

\begin{rem}\rm Let $T\in \PP d$ and $C\subset \PP d$ be a smooth
curve not contained in a hyperplane. It is always true that
$\mathrm{rk}_{C}(T)\leq d$. E.g. if $C$ is the rational normal curve
$C=C_d\subset \PP d$, this maximum value of the rank can be attained
by a tensor $T$, and this is precisely the case when $T$ belongs to
$\tau (C_d) \setminus C_d$, see Theorem \ref{curve}). Actually Theorem
\ref{ellcurve} shows that, if $d=3$, then there are tensors of $\PP
3$ whose rank with respect to an elliptic normal curve
$\Gamma_{4}\subset \PP 3$ is precisely $3$. In the very same way,
one can check that the same is true for a rational (non-normal)
quartic curve $C_4\subset \PP 3$. For the case of space curves,
several other examples can be found in \cite{Pi}.
\end{rem}


\subsection{Simplified version of The Sylvester Algorithm}

Theorem 23 allows to get a simplified version of the Sylvester
algorithm (see also \cite{CS}), which  computes only the symmetric
rank of a symmetric tensor, without computing the actual
decomposition.

\begin{alg}\label{SSRA}
\rm{\textbf{The (Sylvester) Symmetric Rank Algorithm}:
\\
\\
\textbf{Input}: The projective class $T$ of a symmetric tensor $t\in S^{d}V$ with $\dim (V)=2$\\
\textbf{Output}: $\sr(t)$.
\begin{enumerate}
\item Initialize $r=0$;
\item\label{oursyl2} Increment $r\leftarrow r+1$;
\item Compute  $M_{d-r,r}(t)$'s $(r+1)\times (r+1)$-minors; if they
are not all equal to zero then go to step \ref{oursyl2}; else, $T\in
\sigma_r(C_d)$ (notice that this happens for $r\leq
\lceil\frac{d+1}{2}\rceil$); go to step \ref{oursyl4}.
\item\label{oursyl4} Choose a solution $(\overline{u}_{0}, \ldots ,
\overline{u}_{d})$ of the system $M_{d-r,r}(t)\cdot
(u_0,\ldots,u_r)^{t} =0$. If the polynomial
$\overline{u}_0t_0^d+\overline{u}_1t_0^{d-1}t_1+\cdots+\overline{u}_rt_1^r$
has distinct roots, then $\sr (t) = r$, i.e. $T\in
\sigma_{r,r}(C_d)$, otherwise $\sr (t) = d-r+2$, i.e. $T\in
\sigma_{r,d-r+2}(C_d)$.
\end{enumerate}}
\end{alg}


\section{Beyond dimension two}\label{Beyond dimension two}

The maps in (\ref{star})  have to be reconsidered when working on $\PP
n$, $n\geq 2$, and with secant varieties to the Veronese variety
$X_{n,d}\subset \PP N$, $N = {d+n \choose n}-1$.  Now a polynomial
in $K[x_0,\ldots ,x_n]_r$ gives a divisor, which is not a
0-dimensional scheme, so the previous construction would not give
$(r-1)$-spaces which are $r$-secant to the Veronese variety.

Actually in this case, when following the construction in
(\ref{star}), we associate to a polynomial $f \in K[x_0,\ldots
,x_n]_r$, the degree $d$ part of the principal ideal $(f)$, i.e. the
vector space $(f)_d \subset K[x_0,\ldots ,x_n]_d$, which is ${d-r+n
\choose n}$-dimensional. Then, working by duality as before, we get
a linear space in $\PP N$ which has dimension ${d+n \choose n} -
{d-r+n \choose n} -1$ and it is the intersection of the hyperplanes
containing the image $\nu_d(F) \subset \nu_d(\PP n)$ of the divisor
$F = \{f=0\}$ where $\nu_{d}$ is the Veronese map defined in
Notation \ref{nud}.

Since the condition for a point in $\PP N$ to belong to such a space
is given by the annihilation of the maximal minors of the
catalecticant matrix $M_{d-r,r}^{(n)}$, this shows that such minors
define in $\PP N$ a variety which is the union of the linear spaces
spanned by the images of the divisors (hypersurfaces in $\PP n$) of
degree $r$ on the Veronese $X_{n,d}$ (see \cite{Gh}).

\bigskip

In order to consider linear spaces which are $r$-secant to
$X_{n,d}$, we will change our approach by considering the
 Hilbert scheme of 0-dimensional subschemes of degree $r$ in $\PP n$,
$Hilb_r(\PP n)$, instead of $K[x_0,\ldots ,x_n]_r$:

\begin{equation}\label{Hilb}\begin{array}{c}
 Hilb_r(\PP n) \stackrel{\phi}{\dashrightarrow}
 \ \G\left({d+n \choose n}-r,K[x_0,\ldots ,x_n]_d\right)\stackrel{\beta}\rightarrow
 \\
 \\
 \noindent  \stackrel{\beta}\rightarrow {\mathbb G}\left({d+n \choose n}-r-1,\PPP (K[x_0,\ldots ,x_n]_d)\right)
 \stackrel{\alpha}\rightarrow
 {\mathbb G}(r-1,\PPP(K[x_0,\ldots,x_n]_d)^*).
 \end{array}\end{equation}
\medskip
The map $\phi$ in (\ref{Hilb}) sends a scheme $Z$, with $\deg(Z)=r$,
to the vector space $(I_Z)_d$;  it is defined in the open set which
parameterizes the schemes $Z$ which impose independent conditions to
forms of degree $d$. The isomorphism $\beta$ is the identification
between the vectorial and projective Grassmannians, while $\alpha$
is given by duality.

As in the case $n=1$, the final image in the above sequence of maps gives
the $(r-1)$-spaces which are $r$-secant to the Veronese variety in
$\PP N \cong \PPP (K[x_0,\ldots,x_n]_d)^*$; moreover such a space
cuts the image of $Z$ on the Veronese.

\begin{nota}\label{PiZ} \rm{ From now on we will always use the notation $\Pi_{Z}$ to indicate
 the projective linear subspace of dimension $r-1$ in $\PPP (S^{d}V)$, with $\dim(V)=n+1$,
 generated by the image of a $0$-dimensional scheme $Z\subset \PP n$ of degree $r$ via Veronese embedding.}
\end{nota}


\subsection{The chordal varieties to Veronese varieties}\label{The chordal varieties to Veronese varieties}

Here we describe $\sigma_{r}(X_{n,d})$ for $r=2$ and $n,d\geq 1$.
More precisely we give a stratification of $\sigma_{r}(X_{n,d})$ in
terms of the symmetric rank of its elements. We will end with an
algorithm that allows to determine if an element belongs to
 $ \sigma_{2}(X_{n,d}) $ and, if this is the case, to compute $\sr
 (t)$.

We premit a remark that will be useful in the sequel.
\par
\medskip
\begin{rem}\label{variables}
\rm  When a form $f\in K[x_0,\ldots ,x_n]$ can be written using less
variables (i.e. $f\in K[l_0,\ldots,l_m]$, for $l_j\in K[x_0,\ldots
,x_n]_1$, $m<n$) then the symmetric rank of the symmetric tensor
associated to $f$ (with respect to $X_{n,d}$) is the same one as the
one with respect to $X_{m,d}$, (e.g. see \cite{LS}, \cite{LT}). In
particular, when a tensor is such that $T \in \sigma_r(X_{n,d})
\subset \PPP (S^{d}V)$, $\dim (V) = n+1$, then, if $r < n+1$, there
is a subspace $W\subset V$ with $\dim (W )= r$ such that $T \in \PPP
(S^{d}W)$; i.e. the form corresponding to $T$ can be written with
respect to $r$ variables.
\end{rem}

\begin{thm}\label{secante2} Any $T\in \sigma_{2}(X_{n,d})\subset \PPP (S^dV)$,
 with $\dim(V)=n+1$, can only have symmetric rank equal to $1$, $2$ or $d$. More precisely:
$$\sigma_2(X_{n,d}) \setminus X_{n,d} = \sigma_{2,2}(X_{n,d})\cup \sigma_{2,d}(X_{n,d}),$$
moreover $\sigma_{2,d}(X_{n,d})=\tau (X_{n,d})\setminus X_{n,d}$.

Here $\sigma_{2,2}(X_{n,d})$ and $ \sigma_{2,d}(X_{n,d})$ are
defined in Notation \ref{sigmasq} and $\tau(X_{n,d})$ is defined in
Notation \ref{tau}.
\end{thm}

\begin{proof}The theorem is actually a quite direct consequence of Remark \ref{variables} and of Theorem \ref{curve},
 but let us describe the geometry in some detail.
 Since $r=2$, every $Z\in Hilb_2(\PP n)$ is the complete intersection of a line and a quadric,
so the structure of $I_Z$ is well known: $I_Z =
(l_1,\ldots,l_{n-1},q)$, where $l_i \in R_1$, linearly independent,
and $q\in R_2 - (l_1,\ldots,l_{n-1})_2$.

If $T\in \sigma_2(X_{n,d})$ we have two possibilities; either
$\sr (T) =2$ (i.e. $T\in \sigma_2^0(X_{n,2})$), or $\sr (T) >
2$ i.e. $T$ lies on a tangent line $\Pi_Z$ to the Veronese, which is
given by the image of a scheme $Z$ of degree 2, via the maps
(\ref{Hilb}). We can view $T$ in the projective linear space $H
\cong \PP d$ in $\PPP (S^dV)$ generated by the rational normal curve
$C_d \subset X_{n,d}$, which is the image of the line $L$ defined by
the ideal $(l_1,\ldots,l_{n-1})$ in $\PP n$ with $l_{1}, \ldots ,
l_{n-1}\in V^{*}$ (i.e. $L\subset \PP n$ is the unique line
containing $Z$); hence we can apply Theorem \ref{curve} in order to
get that rk$_{C_d}(T) = d$.

Hence, by Remark \ref{variables}, we have $\sr(T)=d$.
\end{proof}

\begin{rem} \rm Let us check that the annihilation of the $(3\times 3)$-minors of the
first two catalecticant matrices, $M_{d-1,1}$ and $M_{d-2,2}$
determines $\sigma_2(X_{n,d})$ (actually such minors are the
generators of $I_{\sigma_2(X_{n,d})}$, see \cite{K}).

Following the construction before Theorem 3.3, we can
notice that the linear spaces defined by the forms $l_i\in V^{*}$ in the
ideal $I_Z$, are such that their coefficients are the solutions of a
linear system whose matrix is given by the catalecticant matrix $M_{d-1,1}$ defined in
Definition \ref{catalecticant} (where the
$a_i$'s are the coefficients of the polynomial defined by $t$);
since the space of solutions has dimension $n-1$, we get
$\mathrm{rk}(M_{d-1,1})=2$. When we consider the quadric $q$ in $I_Z$, instead,
the analogous construction gives that its coefficients are the
solutions of a linear systems defined by the catalecticant  matrix $M_{d-2,2}$, and
the space of solutions has to give $q$ and all the quadrics in
$(l_1,\ldots,l_{n-1})_2$, which are ${n \choose 2}+2n-1$, hence  $\mathrm{rk}(
M_{d-2,2}) = {n+2 \choose 2}-({n \choose 2}+2n) = 2$.
\end{rem}

Therefore we can write down an algorithm to test if an element $T
\in \sigma_{2}(X_{n,d})$ has symmetric rank $2$ or $d$.

\begin{alg}\label{sigma2Xndalg}\rm{\textbf{Algorithm for the symmetric rank of an element of
 $\mathbf{\sigma_{2}(X_{n,d})}$}
\\
\\
\textbf{Input}: The projective class $T$ of a symmetric tensor $t\in  S^{d}V$, with $\dim(V)=n+1$;\\
\textbf{Output}: $T \notin \sigma_{2}(X_{n,d})$, or $T\in
\sigma_{2,2}(X_{n,d})$, or $ T \in \sigma_{2,d}(X_{n,d})$, or
$T\in X_{n,d}$.
\begin{enumerate}
\item Consider the homogeneous polynomial associated to $t$ as in (3.1) and rewrite it with the minimum  possible number of variables
(methods are described in \cite{Ca} or \cite{Ol}), if this number
is 1 then $T\in X_{n,d}$; if it is $>2$ then $T\notin
\sigma_{2}(X_{n,d})$, otherwise $T$ can be viewed as a point in
$\PPP(S^dW) \cong \PP d \subset \PPP( S^{d}V)$, and $\dim (W) =2$,
so go to step \ref{second}.
\item\label{second} Apply the
Algorithm \ref{SSRA} to conclude.
\end{enumerate}}
\end{alg}


\subsection{Varieties of secant planes to Veronese varieties}\label{Varieties of secant planes to any Veronese variety}

In this section we give a stratification of
$\sigma_{3}(X_{n,d})\subset \PP {}(S^{d}V)$ with $\dim(V)=n+1$ via
the symmetric rank of its elements.

\begin{lemma}\label{1} Let $Z \subset \PP {n}$, $n\geq 2$,
be a 0-dimensional scheme, with $\deg(Z) \leq 2d+1$. A
necessary and sufficient condition for $Z$ to impose independent
conditions to hypersurfaces of degree $d$ is that no line
$L\subset \PP {n}$ is such that $\deg (Z\cap L) \geq d+2$.
\end{lemma}

\begin{proof} The statement was probably classically known, we prove it
here for lack of a precise reference. Notice that $h^0({\mathcal O}_{\PP n}(d))={d+n \choose d} \geq 2d+1$, so what we have to prove is that, for $Z$ as in the statement, if there exists no line $L$ such that $\deg L\cap Z \geq d+2$, then $h^1({\mathcal I}_Z(d))=0$.
Let us work by induction on
$n$ and $d$; if $d=1$ the statement is trivial; so let us suppose
that $d\geq 2$. Let us
consider the case $n=2$ first.  If there is a line $L$ which
intersects $Z$ with multiplicity $\geq d+2$, then trivially $Z$
cannot impose independent condition to curves of degree $d$, since
the fixed line imposes $d+1$ conditions, hence we have already
missed one. So, suppose that there exists no such line, and let $L$
be a line such that $Z\cap L$ is as big as possible (hence $2\leq \deg (Z\cap
L)\leq d+1$). Let the $Trace$ of $Z$ on $L$, $Tr_L Z$, be the
schematic intersection $Z\cap L$ and the $Residue$ of $Z$ with
respect to $L$, $Res_L Z$, be the scheme defined by $(I_Z : I_L)$.
Notice that $\deg (Tr_L Z)+\deg (Res_L Z) = \deg Z$.
We have the following exact sequence of ideal sheaves:
$$ 0 \rightarrow {\mathcal I}_{Res_L Z}(d-1) \rightarrow {\mathcal I}_{Z}(d)
\rightarrow {\mathcal I}_{Tr_L Z}(d) \rightarrow 0.$$ Then no line can
intersect $Res_L Z$ with multiplicity $\geq d+1$, because $\deg(Z)
\leq 2d+1$ and $L$ is a line with maximal intersection with $Z$; so
if $\deg(L'\cap res_L Z)=d+1$, we'd have that also $\deg(L\cap
Z)=d+1$, which is impossible because it would give
$2d+1\geq \deg Z\geq \deg (Tr_L Z)+\deg (L'\cap Res_L Z)=2d+2$. Since $\deg (Res_L Z)\leq 2d+1$,
we have $h^1({\mathcal I}_{Res_L Z}(d-1))=0$, by induction
on $d$. On the other hand, we have $h^1({\mathcal I}_{Tr_L Z}(d)) =
h^1({\mathcal O}_{\PP 1}(d - \deg (Tr_L Z))) =0$, hence  also $h^1({\mathcal
I}_{Z}(d))=0$, i.e. $Z$ imposes independent conditions to curves of
degree $d$.

With the case $n=2$ done, let us finish by induction on $n$.
Consider $n\geq 3$, if there is a line $L$ which intersects $Z$ with
multiplicity $\geq d+2$, we can conclude again that $Z$ does not
impose independent conditions to forms of degree $d$, as in the case
$n=2$. Otherwise, consider a hyperplane $H$, with maximum
multiplicity of intersection with $Z$, and consider the exact
sequence:
$$ 0 \rightarrow {\mathcal I}_{Res_H Z}(d-1) \rightarrow {\mathcal I}_{Z}(d)
\rightarrow {\mathcal I}_{Tr_H Z}(d) \rightarrow 0.$$ We have $h^1({\mathcal
I}_{Res_H Z}(d-1))=0$, by induction on $d$, and $h^1({\mathcal I}_{Tr_H
Z}(d))=0$, by induction on $n$, so we get that $h^1({\mathcal
I}_{Z}(d))=0$ again, and we are done.
\end{proof}

\begin{rem}\label{rem to lemma 1} \rm Notice that if $\deg L\cap Z$ is exactly $d+1+k$, then
the dimension of the space of curves of degree $d$ through $Z$
increases exactly by $k$ with respect to the generic case.
\end{rem}

In the sequel we will need the following definition.

\begin{defi}\label{2jet} A {\it $t$-$jet$} is  a $0$-dimensional scheme
$J\subset \PP n$ of degree $t$ with support at a point $P\in \PP n$
and contained in a line $L$; namely the ideal of $J$ is of type:
$I_P ^t + I_L$, where $L\subset \PP n$ is a line containing $P$.  We
will say that $J_1,\ldots ,J_s$ are generic $t$-jets in $\PP n$ if
for each $i=1,...,s$, we have $I_{J_i}= I_{P_i}^t + I_{L_i}$, the
points $P_1,\ldots ,P_s$ are generic in $\PP n$ and
$\{L_1,\dots,L_s\}$ is generic among all the sets of $s$ lines with
$P_i\in L_i$.
\end{defi}

\begin{thm}\label{sec3Xd} Let $d\geq 3$, $X_{n,d}\subset \PPP (S^dV)$. Then:

$\sigma_{3}(X_{n,3})\setminus
\sigma_{2}(X_{n,3})=\sigma_{3,3}(X_{n,3})\cup\sigma_{3,4}(X_{n,3})\cup
\sigma_{3,5}(X_{n,3})$, while, for $d\geq 4$:

$\sigma_{3}(X_{n,d})\setminus
\sigma_{2}(X_{n,d})=\sigma_{3,3}(X_{n,d})\cup\sigma_{3,d-1}
(X_{n,d})\cup\sigma_{3,d+1}(X_{n,d})\cup \sigma_{3,2d-1}(X_{n,d}).$

The $\sigma_{b,r}( X_{n,d})$'s are defined in Notation
\ref{sigmasq}.
\end{thm}

\begin{proof} For any scheme $Z\in Hilb_{3}(\PPP (V))$ there exists
 a subspace $U\subset V$ of dimension $3$ such that $Z\subset \PPP (U)$.
 Hence, when we make the construction in (\ref{Hilb}) we get that $\Pi_{Z}$
 is always a $\PP 2$ contained in $\PPP (S^{d}U)$ and $\nu_{d}(\PPP (U))$ is a
 Veronese surface $X_{2,d}\subset \PPP (S^{d}U)\subset \PPP (S^{d}V)$. Therefore,
 by Remark \ref{variables}, it is sufficient to prove the statement for $X_{2,d}\subset \PPP (S^{d}U)$.

First we will consider the case when there is a line $L$ such that
$Z\subset L$. In this case, let $C_d = \nu_d(L)$; we get that $T\in
\sigma_3(C_d)$, hence either $T\in \sigma_{3,3}(C_d)$ (so $T\in
\sigma_{3,3}(X_{2,d})$), or (only when $d\geq 4$) $T\in
\sigma_{3,d-1}(C_d)$, hence  $\sr (T)\leq  d-1$. The symmetric rank
of $T$ is actually $d-1$ by Remark \ref{variables}.

Now we let $Z$ not to be on a line; the scheme $Z\in Hilb_{3}(\PP
n)$ can have support on $3$ , $2$ distinct points or on one point.

If $Supp(Z)$ is the union of $3$ distinct points then clearly
$\Pi_{Z}$, that is the image of $Z$ via (\ref{Hilb}),  intersects
$X_{2,d}$ in $3$ different points and hence any $T\in \Pi_{Z}$ has
symmetric rank precisely $3$, so $T\in \sigma_{3,3}(X_{2,d})$.

If $Supp(Z)=\{P,Q\}$ with $P\neq Q$, then the scheme $Z$ is the
union of a simple point, $Q$, and of a  $2$-jet $J$ (see Definition
\ref{2jet}) at $P$. The structure of $2$-jet on $P$ implies that
there exists a line $L\subset \PP n$ whose intersection with $Z$ is
a $0$-dimensional scheme of degree $2$. Hence
$\Pi_{Z}=<T_{\nu_{d}(P)}(C_{d}), \nu_{d}(Q)>$ where
$T_{\nu_{d}(P)}(C_{d})$ is the projective tangent line at
$\nu_{d}(P)$ to $C_{d}=\nu_{d}(L)$. Since $T\in \Pi_{Z}$, the line
$<T, \nu_{d}(Q)>$ intersects $T_{\nu_{d}(P)}(C_{d})$ in a point $Q'
\in \sigma_{2}(C_{d})$. From Theorem \ref{curve} we know that
$\sr(Q')=d$. We may assume that $T\neq Q'$ because otherwise $T$
should belong to $\sigma_{2}(X_{2,d})$.

We have $Q\notin L$ because $Z$ is not in a line, so $T$ can be
written as a combination of a tensor of symmetric rank $d$ and a
tensor of symmetric rank $1$, hence $\sr(t)\leq d+1$. If $\sr(t)=d$,
then there should exist $Q_{1}, \ldots , Q_{d}\in X_{2,d}$ such that
$T\in<Q_{1}, \ldots , Q_{d}>$; notice that $Q_1,\ldots,Q_d$ are not
all on $C_d$, otherwise $T\in \sigma_2(X_{2,d})$. Let
$P_1,\ldots,P_d$ be the pre-image via $\nu_{d}$ of $Q_{1}, \ldots ,
Q_{d}$; then $P_1,\ldots,P_d$ together with $J$ and $Q$ should not
impose independent conditions to curves of degree $d$, so, by Lemma
\ref{1}, either $P_1,\ldots,P_d,J$ are on $L$, or
$P_1,\ldots,P_d,P,Q$ are on a line $L'$. The first case is not
possible, since $Q_1,\ldots,Q_d$ are not all on $C_d$. In the other
case notice that, by Lemma \ref{1} and the Remark \ref{rem to lemma
1}, we should have that $<Q_{1}, \ldots ,
Q_{d},T_{\nu_{d}(P)}(C_{d}),\nu_{d}(Q)> \cong \PP {d}$, but since
$<Q_{1}, \ldots , Q_{d}>$ and $<T_{\nu_{d}(P)}(C_{d}),\nu_{d}(Q)>$
have $T,\nu_{d}(P)$ and $\nu_{d}(Q)$ in common, they generate a
$(d-1)$-dimensional space, but this is a contradiction. Hence
$\sr(t)=d+1$.

This construction also shows that $T\in \sigma_{3,d+1}(X_{2,d})$,
and that there exists $W\subset V$ with $\dim(W)=2$ and
$l_{1},\ldots , l_{d}\in W^{*}$ and $l_{d+1}\in V^{*}$ such that
$t=l_{1}^{d}+\cdots +l_{d}^{d}+l_{d+1}^{d}$ and $t=[T]$.

If $Supp(Z)$ is only one point $P\in \PP 2$, then $Z$ can only be
one of the following: either $Z$ is $2$-fat point (i.e. $I_Z$ is
$I_P^2$), or there exists a smooth conic containing $Z$.
\\
If $Z$ is a 2-fat point then $\Pi_{Z}$ is the tangent space to
$X_{2,d}$ at $\nu_{d}(P)$, hence if $T\in \Pi_{Z}$, then the line
$<\nu_{d}(P), T>$ turns out to be a tangent line to some rational
normal curve of degree $d$ contained in $X_{2,d}$, hence in this
case $T\in \sigma_{2}(X_{2,d})$.
\\
If there exists a smooth conic $C\subset \PP 2$ containing $Z$,
write $Z=3P$ and consider $C_{2d}=\nu_{d}(C)$, hence $T\in
\sigma_{3}(C_{2d})$, therefore by Theorem \ref{curve}
rk$_{C_{2d}}(T)$, hence $\sr(t)\leq 2d-1$. Suppose that $\sr(t)\leq
2d-2$, hence there exist $P_{1}, \ldots , P_{2d-2}\in \PP 2$
distinct points that are neither on a line nor on a conic containing
$3P$, such that $T\in \Pi_{Z'}$ with $Z'=P_{1}+ \cdots + P_{2d-2}$
and $Z+Z'=3P+P_{1}+ \cdots + P_{2d-2}$ doesn't impose independent
conditions to the planes curves of degree $d$. Now, by Lemma \ref{1}
we get that $3P+ P_{1}+ \cdots + P_{2d-2}$ doesn't impose
independent conditions to the plane curves of degree $d$ if and only
if there exists a line $L\subset \PP 2$ such that $\deg((Z+ Z')\cap
L)\geq d+2$. Observe that $Z'$ cannot have support contained in a
line because otherwise $T\in \sigma_{2}(X_{2,d})$. Moreover $Z+Z'$
cannot have support on a conic $C\subset \PP 2$ because in that case
$T$ would have rank $2d-1$ with respect to $\nu_{d}(C)=C_{2d}$,
while rk$_{C_{2d}}(T)$.
\\
 Assume that $\deg((Z+Z')\cap L)=d+2$ first; we have to check the following cases:
\begin{enumerate}
\item There exist $P_{1}, \ldots , P_{d+2}\in Z'$ on a line $L\subset \PP 2$;
\item There exist $P_{1}, \ldots , P_{d+1}\in Z'$ such that together with $P=Supp(Z)$ they are on the same line $L\subset \PP 2$;
\item There exist $P_{1}, \ldots , P_{d}\in Z'$ such that together with the $2$-jet $2P$ they are on the same line $L\subset \PP 2$.
\end{enumerate}

\begin{itemize}
\item[\textbf{Case 1.}] Let $P_{1}, \ldots , P_{d+2}\in L\subset \PP 2$, then $\nu_{d}(L)=C_{d}\subset \PP d\subset \PP N$ with $N={d+2 \choose 2}-1$. Clearly $T\in \Pi_{Z}\cap \Pi_{Z'}$, then $\dim(\Pi_{Z}+\Pi_{Z'})\leq \dim(\Pi_{Z})+ \dim(\Pi_{Z'})$, moreover $\Pi_{Z'}$ doesn't have dimension $2d-3$ as expected because $\nu_{d}(P_{1}), \ldots , \nu_{d}(P_{d+2})\in C_{d}\subset \PP d$, hence $\dim(\Pi_{Z'})\leq 2d-4$ and $\dim(\Pi_{Z}+ \Pi_{Z'})\leq 2d-2$. But this is not possible because $Z+Z'$ imposes to the plane curves of degree $d$ only one condition less then the expected, hence $\dim (I_{Z+Z'}(d))={d+1 \choose 2}-d+1$ and then $\dim (\Pi_{Z}+\Pi_{Z'})=2d-1$, that is a contradiction.
\item[\textbf{Case 2.}] Let $P_{1}, \ldots , P_{d+1}, P \in L\subset \PP 2$, then
$\nu_{d}(P_{1}), \ldots , \nu_{d}(P_{d+1}), \nu_{d}(P) \in \nu_{d}(L)=C_{d}$.
Now $\Pi_{Z}\cap \Pi_{Z'}\supset \{\nu_{d}(P), T\}$, then again $\dim(\Pi_{Z}+\Pi_{Z'})\leq 2d-2$.
\item[\textbf{Case 3.}] Let $P_{1}, \ldots , P_{d},2P \in L\subset \PP 2$,
as previously $\nu_{d}(P_{1}), \ldots , \nu_{d}(P_{d+1}),
\nu_{d}(2P) \in \nu_{d}(L)=C_{d}$, then now $T_{\nu_{d}(P)}(C_{d})$
is contained in $<C_{d}>\cap \Pi_{Z}$. Since $<\nu_{d}(P_{1}),
\ldots , \nu_{d}(P_{d})>$ is a hyperplane in $<C_{d}>=\PP d$, it
will intersect $T_{\nu_{d}(P)}(C_{d})$ in  a point $Q$ different
form $\nu_{d}(P)$. Again $\dim(\Pi_{Z}\cap \Pi_{Z'})\geq 1$ and then
$\dim(\Pi_{Z}+\Pi_{Z'})\leq 2d-2$.
\end{itemize}
When $\deg((Z+Z')\cap L)=d+k+1$, $k>1$, we can conclude in the same
way, by using Remark 35.
\end{proof}

Now we are almost ready to present an algorithm which allows to
indicate if a projective class of a symmetric tensor in $\PP {{n+d
\choose d}-1}$ belongs to $\sigma_{3}(X_{n,d})$, and in this case to
determine its rank. Before giving the algorithm we need to recall a
result about $\sigma_3(X_{2,3})$:

\begin{rem}\label{Aronhold} \rm The secant variety $\sigma_3(X_{2,3})\subset \PP
9$ is a hypersurface and its defining equation it is the ``Aronhold
(or Clebsch) invariant" (for an explicit expression see e.g.
\cite{Ot}). When $d\geq 4$, instead, $\sigma_3(X_{2,3})$ is defined
(at least scheme theoretically) by the $(4\times 4)$-minors of
$M_{d-2,2}$, see (Landsberg, Ottaviani, 2009).
\end{rem}

Notice also that there is a very direct and well known way of
getting the equations for the secant variety $\sigma_{s}(X_{n,d})$,
which we describe in the next remark. The problem with this method
is that it is computationally very inefficient, and it can be worked
out only in very simple cases.

\begin{rem}\label{equSec}\rm Let $T=[t]=\left[z_{0}, \ldots , z_{{n+d\choose d}}\right]\in \PP {}(S^{d}(V))$,
 where $V$ is an $(n+1)$-dimensional vector space. $T$ is an element of $\sigma_{s}(X_{n,d})$
 if there exist $P_{i}=[x_{0,i}, \ldots , x_{n,i}]\in \PP n= \PP {}(V)$, $i=1, \ldots ,s$,  and
$\lambda_{1}, \ldots , \lambda_{s}\in K$, such that
$t=\lambda_{1}w_{1}+ \cdots + \lambda_{s}w_{s}$, where
$[w_{i}]=\nu_{d}(P_{i})\in \PP {{n+d\choose d}-1}=\PP {}(S^{d}V)$,
$i=1, \ldots , s$ (i.e. $[w_{i}]=[x_{0,i}^{d}, x_{0,i}^{d-1}x_{1},
\ldots , x_{n,i}^{d}]$).

This can be expressed via the following system of equations:
$$\left\{
\begin{array}{l}
z_{0}=\lambda_{1}x_{0,1}^{d}+ \cdots + \lambda_{s}x_{0,s}^{d}\\
z_{1}=\lambda_{1}x_{0,1}^{d-1}x_{1,1}+ \cdots + \lambda_{s}x_{0,s}^{d-1}x_{1,s}\\
\vdots \\
z_{{n+d\choose d}-1}=\lambda_{1}x_{n,1}^{d}+ \cdots +
\lambda_{s}x_{s,s}^{d}
\end{array}
\right. .$$ Now consider the ideal $I_{s,n,d}$ defined by the above
polynomials in the weighted coordinate ring $$R=K\left[x_{0,1},
\ldots , x_{n,1}; \ldots ; x_{0,s}, \ldots , x_{n,s};\lambda_{1},
\ldots , \lambda_{s}; z_{0}, \ldots , z_{{n+d\choose d}-1}\right]$$
where the $z_{i}$'s have degree $d+1$:
 $$I_{s,n,d}=(z_{0}-\lambda_{1}x_{0,1}^{d}+ \cdots + \lambda_{s}x_{0,s}^{d},
z_{1}-\lambda_{1}x_{0,1}^{d-1}x_{1,1}+ \cdots +
\lambda_{s}x_{0,s}^{d-1}x_{1,s}, \ldots , z_{{n+d\choose
d}-1}-\lambda_{1}x_{n,1}^{d}+ \cdots + \lambda_{s}x_{s,s}^{d}).$$
Now eliminate from $I_{s,n,d}$ the variables $\lambda_{i}$'s and
$x_{j,i}$'s, $i=1, \ldots , s$ and $j=0, \ldots , n$. The
elimination ideal $J_{s,n,d}\subset K\left[z_{0}, \ldots ,
z_{{n+d\choose d}-1}\right]$ that we get from this process is an
ideal of $\sigma_{s}(X_{n,d})$.

Obviously  $J_{s,n,d}$ contains all the $(s+1)\times (s+1)$ minors
of the catalecticant matrix of order $r\times (d-r)$ (if they
exist).
\end{rem}

\medskip
\begin{alg}\rm{
\textbf{Algorithm for the symmetric rank of an element of
$\mathbf{\sigma_{3}(X_{n,d})}$}
\\
\\
\textbf{Input}: The projective class $T$ of a symmetric tensor $t\in  S^{d}V$,  with
$\dim(V)=n+1$;
\\
\textbf{Output}: $T\notin \sigma_{3}(X_{n,d})$ or $T\in
\sigma_{2}(X_{n,d})$ or $T\in \sigma_{3,3}(X_{n,d})$ or $T\in
\sigma_{3,d-1}(X_{n,d})$ or $T\in \sigma_{3,d+1}(X_{n,d})$ or $T\in
\sigma_{3,2d-1}$.
\begin{enumerate}
\item Run the first step of Algorithm \ref{sigma2Xndalg}. If only one
variable is needed, then $T\in X_{n,d}$;  if two variables are
needed, then $T\in \sigma_{3}(X_{n,d})$ and use Algorithm
\ref{sigma2Xndalg} to determine $\sr (T)$. If the number of
variables is greater than $3$, then $T\notin \sigma_{3}(X_{n,d})$.
 Otherwise (three variables) consider $t \in S^d(W)$, with $\dim (W) =3$ and
go to next step;
\item\label{d} If $d=3$, evaluate the Aronhold invariant
(see \ref{Aronhold}) on $T$, if it is zero on $T$ then $T\in
\sigma_{3}(X_{2,3})$ and go to step \ref{t}; otherwise $T\notin
\sigma_{3}(X_{2,3})$. If $d\geq 4$, evaluate rk$M_{2,d-2}(T)$; if
rk$M_{2,d-2}(T)\geq 4$, then $T\notin \sigma_{3}(X_{2,d})$ ;
otherwise $T\in \sigma_{3}(X_{2,d})$ and go to step \ref{t}.
 \item\label{t} Consider the space $S\subset  K[x_0,x_1,x_2]_2$ of the
solutions of the system $M_{2,d-2}(T)\cdot (b_{0,0}, \ldots ,
b_{2,2})^t=0$. Choose three generators $F_{1},F_{2},F_{3}$ of $S$.

\item\label{radical} Compute the radical ideal $I$ of the
ideal $(F_{1},F_{2},F_{3})$ (this can be done e.g. with
\cite{CoCoA}). Since $\dim (W)=3$, i.e. 3 variables were needed,
$F_{1},F_{2},F_{3}$ do not have a common linear factor.

\item\label{nonallin} Consider the generators of $I$. If there are
two linear forms among them, then $T\in \sigma_{3,
2d-1}(X_{n,d})$, if there is only one linear form then $T\in
\sigma_{3,d+1}(X_{n,d})$, if there are no linear forms then $T\in
\sigma_{3,3}(X_{n,d})$.
\end{enumerate}}
\end{alg}


\subsection{Secant varieties of $X_{2,3}$}\label{Secant varieties of $X_{2,3}$}

In this section we describe all possible symmetric ranks that can
occur in $\sigma_{s}(X_{2,3})$ for any $s\geq 1$.

\begin{thm} Let $U$ be a $3$-dimensional vector space. The stratification of the cubic
forms of $\PP {}(S^{3}U^{*})$ with respect to symmetric rank is the
following:
\begin{itemize}
\item $X_{2,3}=\{T\in \PPP ( S^{3}U)\; | \; \sr(T)=1\}$; \item
$\sigma_{2}(X_{2,3})\setminus X_{2,3}= \sigma_{2,2}(X_{2,3})\cup
\sigma_{2,3}(X_{2,3})$; \item $\sigma_{3}(X_{2,3})\setminus
\sigma_{2}(X_{2,3})=\sigma_{3,3}(X_{2,3})\cup\sigma_{3,4}(X_{2,3})\cup
\sigma_{3,5}(X_{2,3})$; \item $\PP 9 \setminus
\sigma_{3}(X_{2,3})=\sigma_{4,4}(X_{2,3})$;
\end{itemize}
where $\sigma_{s,m}(X_{2,3})$ is defined as in Notation
\ref{sigmasq}.
\end{thm}

\begin{proof} We only need to prove that $\PP 9 \setminus \sigma_{3}(X_{2,3})=\sigma_{4,4}(X_{2,3})$ because $X_{2,3}$
is by definition the set of symmetric tensors of symmetric rank $1$
and the cases of $\sigma_{2}(X_{2,3})$ and $\sigma_{3}(X_{2,3})$ are
consequences of Theorem \ref{secante2} and Theorem \ref{sec3Xd}
respectively.

So now we show that all symmetric tensors in $\PP 9 \setminus
\sigma_{3}(X_{2,3})$ are of symmetric rank $4$. Clearly, since they
do not belong to $\sigma_{3}(X_{2,3})$, they have symmetric rank
$\geq 4$; hence we need to show that their symmetric rank is
actually less or equal than $4$.  Let $T\in \PP 9 \setminus
\sigma_{3}(X_{2,3})$ and consider the system $M_{2,1}\cdot
(b_{0,0},\ldots , b_{2,2})^{T}=0$. The space of solutions of this
system gives a vector space of conics which has dimension $3$;
moreover it is not the degree $2$ part of any ideal representing a
$0$-dimensional scheme of degree $3$ (otherwise we'd have $T\in
\sigma_3(X_{2,3}$), hence the generic solution of that system is a
smooth conic. Therefore in the space of the cubics through $T$,
there is a subspace given by $<C\cdot x_{0}, C\cdot x_{1}, C\cdot
x_{2}>$ where $C$ is indeed a smooth conic given by the previous
system. Hence, if $C_{6}$ is the image of $C$ via the Veronese
embedding $\nu_{3}$, we have that $T\in<C_{6}>$, in particular $T\in
\sigma_{4}(C_{6})\setminus \sigma_{3}(C_{6})$, therefore $\sr(t)\leq
6-4+2=4$.
\end{proof}


\subsection{Secant varieties of $X_{2,4}$}\label{Secant varieties of $X_{2,4}$}

We recall that the $k$-th osculating variety to $X_{n,d}$, denoted
by ${\mathcal O}_{k,n,d}$, is the closure of the union of the
$k$-osculating planes to the Veronese variety $X_{n,d}$, where the
$k$-osculating plane ${\mathcal O}_{k,n,d,P}$ at the point $P\in
X_{n,d}$ is the linear space generated by the $k$-th infinitesimal
neighborhood $(k+1)P$ of $P$ on $X_{n,d}$ (see for example
\cite{BCGI} 2.1, 2.2). Hence for example the first osculating
variety is the tangential variety.

\begin{lemma}\label{lemmaosc} The second osculating variety
${\mathcal O}_{2,2,4}$ of $X_{2,4}$ is contained in $\sigma_{4}(X_{2,4})$.
\end{lemma}

\begin{proof} Let $T$ be a generic element of ${\mathcal O}_{2,2,4}\subset \PPP (S^{4}V)$ with $\dim(V)=3$. Hence $T=l^{2}{\mathcal C}$ where $l$ and $\mathcal C$ are a linear and a quadratic  generic forms respectively of $\PPP (S^{4}V)$ regarded as a projectivization of the homogeneous polynomials of degree $4$ in $3$ variables, i.e. $K[x,y,z]_{4}$ (see \cite{BCGI}). We can always assume that $l=x$ and ${\mathcal C}=a_{0,0}x^{2}+a_{0,1}xy+a_{0,2}xz+a_{1,1}y^{2}+a_{1,2}yz+a_{2,2}z^{2}$. The catalecticant matrix $M_{2,2}$ (defined in general in Definition \ref{catalecticant}) for a plane quartic $a_{0000}x^{4}+ a_{0001}x^{3}y+ \cdots + a_{2222}z^{4}$ is the following:
$$M_{2,2}=\left( \begin{array}{cccccc}
a_{0000} & a_{0001} & a_{0002} & a_{0011} & a_{0012} & a_{0022} \\
a_{0001} & a_{0011} & a_{0012} & a_{0111} & a_{0112} & a_{0122}\\
a_{0002} & a_{0012} & a_{0022} & a_{0112} & a_{0122} & a_{0222} \\
a_{0011} & a_{0111} & a_{0112} & a_{1111} & a_{1112} & a_{1122}\\
a_{0012} & a_{0112} & a_{0122} & a_{1112} & a_{1122} & a_{1222}\\
a_{0022} & a_{0122} & a_{0222} & a_{1122} & a_{1222} & a_{2222}
\end{array}\right)$$
hence in the specific case of the quartic above  $l^{2}{\mathcal C}=x^{2}(a_{0,0}x^{2}+a_{0,1}xy+a_{0,2}xz+a_{1,1}y^{2}+a_{1,2}yz+a_{2,2}z^{2})$ it becomes:
$$M_{2,2}(T)=\left( \begin{array}{cccccc}
a_{0000} & a_{0001} & a_{0002} & a_{0011} & a_{0012} & a_{0022} \\
a_{0001} & a_{0011} & a_{0012} & 0 & 0 & 0\\
a_{0002} & a_{0012} & a_{0022} & 0 &0 &0 \\
a_{0011} &0& 0 &0 & 0 & 0\\
a_{0012} & 0 & 0 & 0& 0& 0\\
a_{0022} & 0& 0 & 0 &0& 0
\end{array}\right)
$$ that clearly has rank less or equal than 4.
Since the ideal of $\sigma_{4}(X_{2,4})$ is  generated by the
$(5\times 5)$-minors of $M_{2,2}$, e.g. see (Landsberg, Ottaviani,
2010), we have that ${\mathcal O}_{2,2,4} \subset \sigma_{4}(X_{2,4})$.
\end{proof}

\begin{lemma}\label{quattro allin} If $Z\in Hilb_{4}(\PP 2)$ and $Z $ is
contained in a line, then $r=\sr(T)\leq 4$ for any $T\in \Pi_{Z}$,
where $\Pi_{Z}$ is defined in Notation \ref{PiZ}, and $T$ belongs
either to $\sigma_{2}(X_{2,4})$ or to $\sigma_{3}(X_{2,4})$.
Moreover there exists $W$ of dimension $2$ and $l_{1}, \dots ,
l_{r}\in S^{1}W^{*}$ such that $t=l_{1}^{4}+ \cdots + l_{r}^{4}$
with $r\leq 4$.
\end{lemma}

\begin{proof} If there exists a $2$-dimensional subspace $W\subset V$ with
$\dim(V)=3$ such that $Supp(Z)\subset \PPP (W)$ then any $T\in
\Pi_{Z}\subset \PPP (S^{4}V)$ belongs to $\sigma_{4}(\nu_{4}(\PPP
(W)))\simeq \PP 4$, therefore $\sr(T)\leq 4$. If $\sr(T)=2,4$ then
$T\in \sigma_{2}(X_{2,4})$, otherwise $T\in \sigma_{3}(X_{2,4})$.
\end{proof}

\begin{lemma}\label{quattro su conica} If $Z\subset Hilb_{4}(\PP 2)$
and there exists a smooth conic $C\subset \PP 2$ such that $Z\subset
C$, then any $T\in \Pi_{Z}$, with $T\notin \sigma_{3}(X_{2,4})$, is
of symmetric rank 4 or 6.
\end{lemma}

\begin{proof} Clearly $T\in \sigma_{4}(\nu_{4}(C))$ and $\nu_{4}(C)$ is
 a rational normal curve of degree $8$, then $\sr(T)\leq 6$.
 If $\sharp \{Supp (Z)\}=4$ then $\sr(T)=4$. Otherwise
$\sr(T)$ cannot be less or equal than $5$ because there would exists
a $0$-dimensional scheme $Z' \subset \PP 2$ made of $5$ distinct
points such that $T\in \Pi_{Z'}$, then $Z+Z'$ should not impose
independent conditions to plane curves of degree $4$. In fact by
Lemma \ref{1} the scheme $Z+Z'$ doesn't impose independent
conditions to the plane quartic if and only if there exists a line
$M\subset \PP 2$ such that $\deg((Z+Z')\cap M)\geq 6$. If
$\deg((Z')\cap M)\geq 5$ then $T\in \sigma_{2}(X_{2,4})$ or $T\in
\sigma_{3}(X_{2,4})$. Hence assume that $\deg((Z+Z')\cap M)\geq 6$
and $\deg((Z')\cap M)<5$. Consider first the case $\deg((Z+Z')\cap
M)= 6$. Then $\deg ((Z')\cap M)=4$ and $\deg ((Z)\cap M)=2$. We have
that $\Pi_{Z+Z'}$ should be a $\PP 7$ but actually it is at most a
$\PP 6$ in fact $\Pi_{(Z+Z')\cap M}=\PP 4$ because $<\nu_{4}(M)>=\PP
4$, moreover $T\in\Pi_{Z}\cap \Pi_{Z'}$ hence $\Pi_{Z+Z'}$ is at
most a $\PP 6$. Analogously if $\deg ((Z+Z')\cap M)=7$ (it cannot be
more) one can see that $\Pi_{Z+Z'}$ should have dimension $6$ but it
must have dimension strictly less than $6$.
\end{proof}

\begin{thm} The $s$-th secant varieties to $X_{2,4}$, up to $s=5$, are described in terms of symmetric ranks as follows:
\begin{itemize}
\item $X_{2,4}=\{T \in S^{4} V\; | \; \sr(T)=1\}$;
\item $\sigma_{2}(X_{2,4})\setminus X_{2,4}= \sigma_{2,2}(X_{2,4}) \cup \sigma_{2,4}(X_{2,4})$;
\item $\sigma_{3}(X_{2,4})\setminus \sigma_{2}(X_{2,4})=\sigma_{3,3}(X_{2,4})\cup\sigma_{3,5}(X_{2,4})\cup \sigma_{3,7}(X_{2,4})$;
\item $\sigma_{4}(X_{2,4})\setminus \sigma_{3}(X_{2,4})=\sigma_{4,4}(X_{2,4})\cup \sigma_{4,6}(X_{2,4})\cup \sigma_{4,7}(X_{2,4})$;
\item $\sigma_{5}(X_{2,4})\setminus \sigma_{4}(X_{2,4})=\sigma_{5,5}(X_{2,4})\cup \sigma_{5,6}(X_{2,4})\cup \sigma_{5,7}(X_{2,4})$.
\end{itemize}
\end{thm}

\begin{proof}  By definition of $X_{n,d}$ we have that $X_{2,4}$ is the variety parameterizing symmetric tensors of $S^{4}V$ having
 symmetric rank $1$ and the cases of $\sigma_{2}(X_{2,4})$ and $\sigma_{3}(X_{2,4})$ are consequences of Theorem
 \ref{secante2} and Theorem \ref{sec3Xd} respectively.

Now we study  $\sigma_{4}(X_{2,4})\setminus \sigma_{3}(X_{2,4})$.
Let $Z\in Hilb_{4}(\PP 2)$ and $T\in \Pi_{Z}$ be defined as in
Notation \ref{PiZ}.
\begin{itemize}
\item Let $Z$ be contained in a line $L$; then by  Lemma \ref{quattro allin} we have that $T$ belongs
either to $\sigma_{2}(X_{2,4})$ or to $\sigma_{3}(X_{2,4})$.
\item Let $Z\subset C$, with $C$ a smooth conic. Then by Lemma \ref{quattro su conica},
$T\in \sigma_{4,4}(X_{2,4})$ or $T\in \sigma_{4,6}(X_{2,4})$.
\item If there are no smooth conics containing $Z$ then either there is a line $L$ such
that $\deg(Z\cap L)=3$,  or $I_{Z}$ can be written as $(x^{2}, y^{2})$.
We study separately those two cases.
\begin{enumerate}
\item  In the first case the ideal of $Z$ in degree 2 can be written either as $<x^{2}, xy>$ or $<xy,xz>$.
\\
\\
If $(I_{Z})_{2}=<x^{2}, xy>$ then it can be seen that the catalecticant matrix of $T$ is
$$M_{2,2}(T)=\left( \begin{array}{cccccc}
0 & 0& 0 & 0 & 0 & 0 \\
0 & 0& 0 & 0 & 0 & 0\\
0 & 0& 0 & 0 & 0 & a_{0222} \\
0 & 0& 0 & a_{1111} & a_{1112} & a_{1122}\\
0 & 0& 0 & a_{1112} & a_{1122} & a_{1222}\\
0 & 0& a_{0222} & a_{1122} & a_{1222} & a_{2222}
\end{array}\right).$$
Hence, for a generic such $T$, we have that $T\notin
\sigma_{3}(X_{2,4})$ since the rank of $M_{2,2}(T)$ is $4$, while it
has to be 3 for points in $\sigma_{3}(X_{2,4})$. In this case if $Z$
has support in a point then $I_{Z}$ can be written as $(x^{2},xy,
y^{3})$ and the catalecticant matrix defined in Definition
\ref{catalecticant} evaluated in $T$ turns out to be:
$$M_{2,2}(T)=\left( \begin{array}{cccccc}
0 & 0 & 0 & 0 & 0 & 0\\
0 & 0 & 0 & 0 & 0 & 0\\
0 & 0 & 0 & 0 & 0 & a_{0222}\\
0 & 0 & 0 & 0 & 0 & a_{1122}\\
0 & 0 & 0 & 0 & a_{1122} & a_{1222}\\
0 & 0 & a_{0222} & a_{1122} & a_{1222} &
a_{2222}\end{array}\right)$$ that clearly has rank less or equal
than 3. Hence $T\in \sigma_{3}(X_{2,4})$.
\\
Otherwise $Z$ is either made of two 2-jets or one 2-jet and two
simple points. In both cases denote by $R$ the line $y=0$. We have
$\deg(Z\cap R)=2$. Thus $\Pi_{Z}$ is the sum of the linear space
$\Pi_{Z\cap L} \simeq \PP 2$ and $\Pi_{Z\cap R}\simeq \PP 1$. Hence
$T=Q+Q'$ for suitable $Q\in \Pi_{Z\cap L}$ and $Q'\in \Pi_{Z\cap
R}$. Since $Q\in \sigma_{3}(\nu_{4}(L))$  and $Q'$ is in a tangent
line to $\nu_{4}(R)$ we have that $\sr(T)\leq 7$. Working as in
Lemma \ref{quattro su conica} we can prove that $\sr(T)=7$.
\\
\\
Eventually if $(I_{Z})_{2}$ can be written as $(xy, xz)$ then $Z$ is
made of a subscheme $Z'$ of degree 3 on the line $L$ and a simple
point $P\notin L$. In this case $\sr(T)=4$  since $\Pi_Z =
<\Pi_{Z'},\nu_{4}(P)>$ and any element in $\Pi_{Z'}$ has symmetric
rank $\leq 3$ (since it is on $\sigma_{3}(\nu_{4}(L))$).
\item In the last case we have that $I_{Z}$ can be written as $(x^{2}, y^{2})$.
If we write the catalecticant matrix defined in Definition
\ref{catalecticant} evaluated in $T$ we get the following matrix:
$$M_{2,2}(T)=\left( \begin{array}{cccccc}
0 & 0 & 0 & 0 & 0 & 0\\
0 & 0 & 0 & 0 & 0 & a_{0122}\\
0 & 0 & 0 & 0 & a_{0122} & a_{0222}\\
0 & 0 & 0 & 0 & 0 & 0\\
0 & 0 & a_{0122} & 0 & 0 & a_{1222}\\
0 & a_{0122} & a_{0222} & 0 & a_{1222} &
a_{2222}\end{array}\right).$$ Clearly if $a_{0122}=0$ the rank of
$M_{2,2}(T)$ is three, hence such a $T$ belongs to
$\sigma_{3}(X_{2,4})$, otherwise we can make a change of coordinates
(that corresponds to do a Gauss elimination on $M_{2,2}(T)$) that
allows to write the above matrix as follows:
$$M_{2,2}(T)=\left( \begin{array}{cccccc}
0 & 0 & 0 & 0 & 0 & 0\\
0 & 0 & 0 & 0 & 0 & a_{0122}\\
0 & 0 & 0 & 0 & a_{0122} & 0\\
0 & 0 & 0 & 0 & 0 & 0\\
0 & 0 & a_{0122} & 0 & 0 & 0\\
0 & a_{0122} & 0& 0 & 0& 0\end{array}\right).$$
This matrix is associated to a tensor $t\in S^{4}V$, with $\dim (V)=3$, that can be written
as the polynomial $t(x_{0},x_{1},x_{2})=x_{0}x_{1}x_{2}^{2}$. Now $\sr(t)=6$ (see \cite{LT}, Proposition 11.9).
\end{enumerate}
\end{itemize}

\a We now study $\sigma _5(X_{2,4})\setminus \sigma _4(X_{2,4})$, so
in the following we assume $T\notin \sigma _4(X_{2,4})$, which
implies $\sr(T) \geq 5$. We have to study the cases with
$\deg(Z)=5$, i.e., $Z\in Hilb_5(\PP 2)$. The scheme $Z$ is hence
always contained in a conic, which can be a smooth conic, the union
of 2 lines or a double line. In the last two cases, $Z$ might be
contained in a line; we now distinguish the various cases according
to these possibilities.

\begin{itemize}
\item $Z$ is contained in a line $L$: $\Pi_{Z}\cong \PP 4$ is spanned by the
rational curve $\nu (L)=C_4$, hence $\sr(T) \leq 4$, against assumptions.

\item $Z$ is contained in a smooth conic $C$. Hence $ \Pi_{Z}$ is spanned by the
subscheme $\nu (Z)$ of the rational curve $\nu (C)=C_8$, so that $T\in \sigma _5(C_8)$ and by Theorem \ref{curve} $\sr(T) =5$.

\item  $Z$ is contained in the union of two lines $L$ and $R$. We say that $Z$ is
of type $(i,j)$ if $\deg (Z\cap L)=i $ and $\deg (Z\cap R)=j $ and
for any other couple of lines in the ideal of $Z$ the degree of the
intersections is not smaller.  Four different cases can occur: $Z$
is of type $(3,2)$, in which case $Z\cap L \cap R= \emptyset$, $Z$
is of type $(3,3)$ or $(4,2)$, and in these two cases $Z$, $L$ and
$R$ meet in a point $P$, $Z$ is of type $(4,1)$, in which case $R$
is not unique. We set $C_4=\nu (L)$, $C'_4=\nu (R)$, $O=\nu (P)$,
$\Pi _L=<\nu (Z\cap L)>$ and $\Pi _R=<\nu (Z\cap R)>$.

\begin{itemize}
\item $Z$ is of type $(4,1)$. Hence $\Pi _Z$ is sum of the linear space
$\Pi _L\subseteq \sigma _4(C_4)$ and the point $Q=\Pi _R\in
X_{2,4}$, so that $T=Q'+Q$ for a suitable $Q'\in  \sigma _4(C_4)$,
and since $\sr(Q')\leq 4$ by Theorem \ref{curve}, we get
$\sr(Q')\leq 5$ .

\item $Z$ is of type $(3,2)$. Hence $\Pi _Z$ is sum of the linear spaces
 $\Pi _L\cong \PP 2$ and the line  $\Pi _R$, so that $T=Q'+Q$ for suitable
  $Q\in \Pi _L\subseteq \sigma _3(C_4)$ and $Q'\in \Pi _R\subseteq \sigma _2(C'_4)$.
  Since  $\sr(Q)\leq 3$ and $\sr(Q')\leq 4$, we get  $\sr(Q)\leq 7$.

\item $Z$ is of type $(3,3)$. Hence $\Pi _Z$ is sum of the linear spaces
 $\Pi _L\cong \PP 2$ and $\Pi _R\cong \PP 2$ meeting at one point, so that $T=Q'+Q$
 for suitable $Q\in \Pi _L\subseteq \sigma _3(C_4)$ and $Q'\in \Pi _R\subseteq \sigma _3(C'_4)$.
 Since  $\sr(Q)\leq 3$ and $\sr(Q')\leq 3$, we get  $\sr (T)\leq 6$. Moreover if $Z$ has
 support on $4$ points, we see that $\sr(T)=6$, using the same kind of argument as in Lemma \ref{quattro su conica}.

\item $Z$ is of type $(4,2)$.  In this case $(I_{Z})_{2}$ can be written as $<xy,x^{2}>$,
then working as above we can see that the catalecticant matrix $M_{2,2}(T)$ has rank $4$.
Since at least set theoretically $I(\sigma_{4}(X_{2,4}))$ is generated by the $5\times 5$
minors of $M_{2,2}$, we conclude that such $T$ belongs to $\sigma_{4}(X_{2,4})$.
\end{itemize}

\item $Z$ is contained in a double line. We distinguish the following cases:
\begin{itemize}
\item The support of $Z$ is a point $P$, i.e. the ideal of $Z$ is
either of type  $(x^3,x^2y,y^2)$ or, in affine coordinates,
$(x-y^{2},y^{4})\cap (x^{2},y)$. In the first case $Z$ is contained
in the 3-fat point supported on $P$, so that $\Pi_{Z}$ is contained
in in the second osculating variety and by Lemma \ref{lemmaosc}
$T\in \sigma_{4}(X_{2,4})$.
\\
In the second case it easy to see that the homogeneous ideal
contains $x^{2}$, $xy^{2}$ and $y^{4}$ and this fact forces the
catalecticant matrix $M_{2,2}(T)$ to have rank smaller or equal to
4. Hence $T\in \sigma_{4}(X_{2,4})$.
 \item The support of $Z$ consists of two points, i.e. the ideal of $Z$ is
 of type  $(x^2,y^2)\cap (x-1,y)$ or $(x^2, xy,y^2)\cap (x-1,y^2)$.
 \par In the first case $Z$ is union of a scheme $Y$ of degree 4  and of a point
  $P$, hence $\Pi _Z$ is sum of the linear spaces $\Pi _Y$ and $\Pi _P$, so that
   $T=Q+\nu (P)$ for suitable $Q\in \Pi _Y$. The above description of the case
   corresponding to $I_{Z}$ of the type $(x^2,y^2)$ shows that either $Q\in \sigma_{3}(X_{2,4})$ or $\sr(Q)=6$.
   Now if $Q\in \sigma_{3}(X_{2,4})$ then clearly $T\in \sigma_{4}(X_{2,4})$, if $\sr(Q)=6$ then $\sr(T)=7$.
 \par
  In the second case  $Z$ is union of a jet and of a 2-fat point, hence $\Pi _Z$ is sum of two linear spaces,
  each of them is contained in a tangent space of $X_{2,4}$ at a different point, so that
  $T=Q+Q'$ with $Q$, $Q'$ contained in the tangential variety; then both $Q$ and $Q'$
  belongs to $\sigma_{2}(X_{2,4})$ hence  $T\in \sigma_{4}(X_{2,4})$.
\item The support of $Z$ consists of three points, i.e. the ideal of $Z$ is of type
$(x,y)\cap ((x^2-1),y^2)$.  Let $P_1, P_2,P_3$ be the points supporting $Z$, with
$\eta _1,\eta _2$ jets such that $Z=\eta _1 \cup \eta _2 \cup P_3$.
There exists a smooth conic $C$ containing $\eta _1 \cup \eta _2 $, and $\nu (C)$ is a $C_8$.
 Then  $\Pi _Z$ is the sum of $\nu (P_3)$ and of the linear space $<\nu (\eta _1),\nu (\eta _2)>$,
 so that $T=Q+\nu (P_3)$ for a suitable $Q \in \sigma _4 (C_8)$, with $\sr(Q)\leq 6$, so we get  $\sr (T)\leq 7$.
 \end{itemize}
 \end{itemize}
\end{proof}

\textbf{Acknowledgements}:
The authors would like to thank E. Ballico, J. M. Landsberg, L.
Oeding and G. Ottaviani for several useful talks and the anonymous
referees for their appropriate and accurate comments.

\end{document}